\documentclass[a4paper,11pt,final]{amsart} 
\usepackage{amssymb,amsmath}
\usepackage{color}
\usepackage{a4wide}

\numberwithin{equation}{section}

\newtheorem{theorem}{Theorem}[section]
\newtheorem{proposition}[theorem]{Proposition}
\newtheorem{corollary}[theorem]{Corollary}
\newtheorem{lemma}[theorem]{Lemma}
\newtheorem{definition}[theorem]{Definition}
\newtheorem{remark}[theorem]{Remark}
\newtheorem{example}[theorem]{Example}

\newcommand{\N}{{\mathbb N}}

\newcommand{\R}{{\mathbb R}}
\newcommand{\Rn}{{\mathbb R}^N}
\newcommand{\Sn}{{\mathbb S}^N}

\def\ol{\overline}
\def\fr{\frac}
\newcommand{\vep}{\varepsilon}

\def\grad{\nabla}

\def\bye{\end{document}}
\def\by{\end{proof}\bye}
\def\pl{\partial}
\def\gO{\Omega}\def\erf{\eqref}\def\tim{\times}

\def\ep{\varepsilon}
\def\gd{\delta}
\def\dist{\mathrm{dist}}
\def\cP{\mathcal P}
\def\beq{\begin{equation}}
\def\eeq{\end{equation}}
\def\cC{\mathcal C}

\newcommand{\Pmo}{\mathcal{P}^-_{1}}
\newcommand{\Ppo}{\mathcal{P}^+_{1}}
\newcommand{\Pmk}{\mathcal{P}^-_{k}}
\newcommand{\Ppk}{\mathcal{P}^+_{k}}

\begin{document}
\parindent=0pt

\title[\textbf{A family of degenerate elliptic operators}]{\textbf{A family of degenerate elliptic operators: maximum principle and its consequences.}}

\author[I. Birindelli, G. Galise, H. Ishii]{Isabeau Birindelli, Giulio Galise, Hitoshi Ishii}

\thanks{ The work was started while IB was visiting HI at the Waseda University, she wishes to thank the institution for the invitation. HI wishes to thank 
Dr. Norihisa Ikoma for his interest 
in Lemma \ref{HH2}. IB and GG were partially supported by GNAMPA-INDAM.
The work of HI was partially supported by the JSPS grants: KAKENHI \#16H03948,  
\#26220702}

\address[I. Birindelli, G. Galise]{Dipartimento di Matematica "G. Castelnuovo"\newline
\indent Sapienza Universit\`a  di Roma \newline
 \indent   P.le Aldo  Moro 2, I--00185 Roma, Italy.}
 \email{isabeau@mat.uniroma1.it, galise@mat.uniroma1.it}
\address[H. Ishii]{Faculty of education and Integrated Arts and Sciences\newline
\indent Waseda University\newline
 \indent 1-6-1 Nishi-Waseda, Shinjuku, Tokyo 169-8050 Japan}
\email{hitoshi.ishii@waseda.jp}

\begin{abstract}
In this paper we investigate the validity and the consequences of  the maximum principle for degenerate elliptic 
operators whose higher order term is the sum of $k$ eigenvalues of the Hessian. In particular we shed some light on 
some very unusual phenomena due to the degeneracy of the operator.
We prove moreover Lipschitz  regularity results and  boundary estimates under convexity assumptions on the 
domain.  As a consequence we obtain the existence of solutions of the Dirichlet problem and of principal eigenfunctions.
 \end{abstract}
\maketitle

\section{Introduction}\label{intro} 
In this paper we shall study solutions of Dirichlet problem for  degenerate elliptic 
operators whose higher order term is given by some sort of \lq\lq truncated Laplacian\rq\rq, i.e.

$$\Pmk(D^2u)=\sum_{i=1}^k\lambda_i(D^2u)\quad \mbox{and}\quad  \Ppk(D^2u)=\sum_{i=N-k+1}^N\lambda_i(D^2u),$$
where $\lambda_1(D^2u)\leq \lambda_2(D^2u)\leq\cdots\leq \lambda_N(D^2u)$ are the ordered eigenvalues of the Hessian of $u$, which have lately been investigated in various contexts e.g. \cite{AS}, \cite{CLN1, CLN3}, \cite{CDLV}, \cite{HL1, HL2}, \cite{Sha}, \cite{Wu}. 
We are interested in the case $N\geq 2$ and  $k<N$ since $\mathcal{P}^-_{N}(D^2u)=\mathcal{P}^+_{N}(D^2u)=\Delta u$. In the whole paper solutions are meant in the viscosity sense, see e.g. \cite{CIL}  and Definition \ref{visc}. 

Clearly, for any symmetric matrix $X$, 
$
\Ppk(X)=-\Pmk(-X)
$
hence we will mainly  state the results for $\Pmk$ with obvious equivalents when the operator $\Ppk$ is considered.
Such operators are positively homogeneous of 
degree one and degenerate elliptic. 

In the following we propose to consider the Dirichlet problem
\begin{equation}\label{basic_equation}
\left\{\begin{array}{lc}
{\mathcal P}^\pm_k(D^2u)+H(x,\grad u)+\mu u=f(x) & \text{in }\ \Omega\\
u=0 & \text{on }\ \partial \Omega,
\end{array}
\right.
\end{equation}    
where $\Omega$ is a bounded domain of $\Rn$ and  the Hamiltonian $H\in C(\Omega\times\Rn;\R)$ is assumed to satisfy the structure condition:
\begin{equation}\label{SC1}
\exists\, b\in\R_+\,\,\text{s.t.}\,\,\left|H(x,\xi)\right|\leq b\left|\xi\right|\quad\forall(x,\xi)\in\Omega\times\Rn.
\tag{SC 1}
\end{equation}
The prototype we have in mind is $H(x,\grad u)=b(x)|\grad u|$ or $H(x,\grad u)=b(x)\cdot\grad u$ with 
$b(x)$ bounded continuous function in $\Omega$. 

In particular we want to raise and partially answer  the following questions, which are very intertwined:
\begin{enumerate}
\item Under which conditions do the operators ${\mathcal P}^\pm_k(D^2u)+H(x,\grad u) +\mu u$ satisfy the  maximum principle, be it weak or strong?
\item What are the regularity of the solutions of the Dirichlet problem?
\item Do the principal eigenvalues and corresponding eigenfunctions exist?
\end{enumerate}

In order to be more specific, let us describe what we call maximum or minimum principle in the sense of the \emph{sign propagation property}.

\begin{definition} $F$ satisfies the maximum or weak maximum principle in $\Omega$ if
$$F[u]\geq 0\ \mbox{in}\ \Omega,\quad \limsup_{x\to\partial\Omega} u\leq 0 \ \Longrightarrow \  u\leq 0 \ \mbox{in}\ \Omega.$$
It satisfies the strong maximum principle if 
$$F[u]\geq 0\ \mbox{in}\ \Omega,\quad u\leq 0 \ \mbox{in}\ \Omega\ \Longrightarrow\ \mbox{either}\ u<0 \ \mbox{or}\ u\equiv 0.$$

Respectively,  $F$ satisfies the minimum or weak minimum principle in $\Omega$ if
$$F[u]\leq 0\ \mbox{in}\ \Omega,\quad  \liminf_{x\to\partial\Omega} u\geq 0 \ \Longrightarrow \ u\geq 0 \ \mbox{in}\ \Omega.$$
It satisfies the strong minimum principle if 
$$F[u]\leq 0\ \mbox{in}\ \Omega,\quad u\geq 0 \ \mbox{in}\ \Omega\  \Longrightarrow\ \mbox{either}\ u>0 \ \mbox{or}\ u\equiv 0.$$
\end{definition}
Of course when $F$ is odd then the notions of maximum and minimum principle are equivalent, but here we shall see that they differ quite a lot.

Just to give a flavour of the kind of results that we shall obtain, let us begin by saying that for any $k<N$, the Hopf 
Lemma, the Harnack inequality and the strong minimum principle do not hold 
in general for solutions of (\ref{basic_equation}). 
On the other hand,  if $bR\leq k$, the weak minimum principle holds in any domain $\Omega\subset B_R$. 
For subsolutions, instead, the strong maximum principle will be a consequence of the Hopf Lemma. The condition $bR\leq k$ has been shown to be optimal in a previous work of the second named author with  Vitolo \cite{GV}.
Other phenomena which are unusual with respect to the uniformly elliptic case will be described in subsection \ref{unusual}.

Historically, the maximum (or minimum) principle for degenerate elliptic operators has been mostly studied when the 
degeneracy depends on the points where the operator acts, e.g.
$$Lu={\rm tr}(A(x)D^2u)\ \mbox{with}\ A\geq 0 $$
 or 
$$Lu=\sum_{i=1}^k X_i^2u,$$
where the $X_i$ are vector fields that may fail to generate the whole space, see e.g. the fundamental 
works of Bony \cite{Bo} or Kohn and Nirenberg \cite{KN}.  We shall not even try to enumerate the results in these sub-elliptic contexts. 

Other class of degenerate operators are the quasilinear operators such as the $p$-Laplacian or the 
$\infty$-Laplacian, whose degeneracy depends on the solution itself, but more precisely on the gradient of the 
solution. 
Here also, for the truncated Laplacian,  the \lq\lq direction\rq\rq of the degeneracy depends on the solution but it 
depends on the  eigenvectors of the Hessian of the solution. 
Let us furthermore remark that these operators are neither linear nor variational. 

The operators $\mathcal{P}^\pm_{k}$ have been initially introduced in connection with Riemannian manifolds. In 
particular when the  
manifolds are $k$ convex this was studied by Sha in 
 \cite{Sha}, the case of partially positive curvature was seen by Wu in \cite{Wu}. Later they can be found in \cite[Example 1.8]{CIL}, as examples of fully nonlinear degenerate elliptic operators,
and \cite{AS}, where Ambrosio and Soner have investigated 
the mean curvature flow with arbitrary codimension through a level set approach. More recently, in a PDE context,
we wish to recall the works of Harvey and Lawson  \cite{HL1, HL2} that have given a new geometric interpretation of 
solutions,  while Caffarelli, Li and Nirenberg in \cite{CLN1, CLN3} in their study of degenerate elliptic equations, give 
some  results concerning removable singularities along smooth manifolds for Dirichlet problems associated to $\Pmk$. See also \cite{AGV} for the extended version of the maximum principle and \cite{CDLV} in the case of entire solutions.

In order to describe the results contained in this work let us introduce the generalized principal eigenvalues
\`a la Berestycki, Nirenberg, Varadhan \cite{BNV}.
For the following equation
\begin{equation}\label{eqev}
\Pmk(D^2u)+H(x,\grad u)+\mu u = 0 \ \mbox{in}\ \Omega,
\end{equation}
we define the following \lq\lq generalized principal eigenvalues\rq\rq:
$$\overline{\mu}_k^+=\sup\{\mu\in\R: \exists w>0 \ \mbox{in}\ \overline\Omega   \mbox{ a supersolution of (\ref{eqev})}\},$$
$${\mu}_k^+=\sup\{\mu\in\R: \exists w>0 \ \mbox{in}\ \Omega   \mbox{  a supersolution of (\ref{eqev})}\}$$
and
$$\overline{\mu}_k^-=\sup\{\mu\in\R: \exists w<0 \ \mbox{in}\ \overline\Omega   \mbox{ a subsolution of (\ref{eqev})}\},$$
$${\mu}_k^-=\sup\{\mu\in\R: \exists w<0 \ \mbox{in}\ \Omega   \mbox{ a subsolution of (\ref{eqev})}\}.$$

When we say that $w$ is a supersolution of \eqref{eqev} and $w>0$ in $\ol\gO$ as in the definition of $\ol\mu_k^+$ above, it is implicit that the function $w$ is defined, 
as a real-valued function, and lower semicontinuous in $\ol\gO$. Similar  
assumptions are made in the definition of $\ol\mu_k^-$ above.

It is immediate that $\overline{\mu}_k^\pm\leq{\mu}_k^\pm$ and also, using (\ref{suba}), that ${\mu}_k^-\leq {\mu}_k^+$ and $\overline{\mu}_k^-\leq \overline{\mu}_k^+$ if $H$ is odd in the gradient.
What we prove in section \ref{demi-eigen} is that these values are 
thresholds
for the validity of the weak maximum or the weak minimum principle,
precisely below  $\overline{\mu}_k^-$ and below $\overline{\mu}_k^+$   the minimum principle and respectively the maximum principle holds.  

In order to be able to reach the values ${\mu}_k^+$ and ${\mu}_k^-$, which are the standard upper bounds in the uniformly elliptic case, we shall need some further conditions.
Precisely, if $\Omega\subset B_R$ with $bR<k$ the maximum principle holds  for any $\mu$ since, we prove in Proposition~\ref{infinito} that $\overline{\mu}_k^+={\mu}_k^+=+\infty$. 
For the minimum principle the situation is more delicate. The weak minimum principle holds up to ${\mu}_k^-$ if,  
beside the above condition on $R$, we shall require that
$\Omega$  satisfies a convexity type assumption, precisely that it is the intersection of a family of balls of same radius; in that case we say that $\Omega$ is a \lq\lq hula hoop\rq\rq domain. 
In particular a $C^2$  strictly convex domain is a hula hoop domain, see Proposition~\ref{hulahoop} . 

Under these hypotheses, in Proposition~\ref{supersol}, we prove that for $\mu=\overline{\mu}^-_k$ the minimum 
principle does not hold. This implies also that $\overline{\mu}^-_k= \mu^-_k$, see Theorem \ref{thm3}; let us emphasize 
that the hula hoop condition does not imply the regularity of the domain e.g. the intersection of two balls of same 
radius.
In general the question of whether $\mu^-_k$ and $\overline\mu_k^-$ coincide is an open problem.

\medskip
In the recent paper \cite{BCPR} that had a great influence on this research, Berestycki, Capuzzo Dolcetta, Porretta and Rossi have studied the validity of the maximum principle for degenerate elliptic operators. For that aim they 
introduce another value 
$$\mu^*:=\sup\{\mu\in\R: \exists \Omega^\prime\supset \overline\Omega,  w>0 \ \mbox{in}\ \Omega^\prime, F[u]+\mu u\leq 0 \ \mbox{in} \ \Omega^\prime \}.$$
Observe that for $F[u]:=\Pmk(D^2u)+H(x,\grad u)$, the value $\mu^*\leq\overline{\mu}_k^+$. In \cite{BCPR} they prove that $F[\cdot]+\mu \cdot$ satisfies the maximum principle in $\Omega$ in the viscosity sense if and only if
$\mu<\mu^*$.   In section 4 of that paper, they also study the equality between the different definitions of generalized principal eigenvalues, but the sufficient conditions require that the domain be regular.

\medskip

The existence of solutions for Dirichlet problems are proved in Section \ref{existence} when $\Omega$ is a hula hoop domain. When the operators concern $\Pmk$ for general $k$, the existence  and uniqueness is given provided that the Hamiltonian is Lipschitz in the gradient variable and 
$\mu<\mu_{k,b}^-\leq\mu_{k}^-$, where $\mu_{k,b}^-$  refers to the generalized principal eigenvalue of $\Pmk(D^2\cdot)-b|\grad \cdot|$ with $b$ the Lipschitz constant of $H$. Instead, thanks to  the Lipschitz estimates, for $k=1$ the existence is given without the extra condition on the Hamiltonian and for any $\mu<\mu_1^-$. In the particular case $f\leq 0$ the existence holds  for any $\mu$. Some  questions concerning existence remain open, e.g.  does the existence of solutions holds when $\mu>\mu_1^-$ for a more general class of forcing terms $f$?  is the hula hoop condition optimal?

Of course a natural question is whether these generalized principal eigenvalues correspond to an eigenfunction. 
In the case of uniformly elliptic fully nonlinear operators, this has been proved to be 
the case in different context (see \cite{Ar, BD, BEQ, IY, L, QS}).
We  are able to give a positive answer to this question when $k=1$ and $\Omega$ is a hula hoop 
domain. This will be somehow an application of the global Lipschitz results that are proved in section \ref{compactness}. The proof of the Lipschitz regularity is extremely sleek.

\medskip
It is quite clear that there are a number of open problems.  
Maybe the most important one is whether the global Lipschitz or H\"older regularity of the solutions holds also for $k\geq 2$. This would in 
particular lead to the existence of the principal eigenfunction in that case as well.  
On one hand it is not surprising that the case of $\Pmo$ is simpler since, when the lower order term is  zero,
solutions of $\Pmo(D^2u)=f(x)$ are semiconvex.
On the other hand, it is also the most degenerate of these operators, so it would be very surprising that the case 
$k=1$ and the case $k=N$ give rise to smooth solutions and that it is not the case for the values of $k$ in between.

Still concerning the regularity, let us recall that 
in the context of convex analysis, Oberman and Silvestre in \cite{OS} prove the $C^{1,\alpha}$ regularity of solutions
of
$$
\Pmo(D^2u)=0 \ \mbox{in } \Omega,\ 
u(x)=g(x) \ \mbox{on } \partial\Omega,$$
under some regularity condition on $g$.
The solution of this problem is 
the convex envelope, of given boundary data $g$.
They proved that the solutions of the Dirichlet problem with $C^{1,\gamma}$ boundary data, 
are $C^{1,\gamma}$ in the interior. When $f$ is not zero and there is a first order term the question 
of the H\"older regularity of the gradient is to our knowledge completely open.

  In the next section, beside recalling a few standard facts, we give estimates near the boundary that will be crucial along the paper. In section \ref{compactness}, using those bounds, we prove global Lipschitz regularity of solutions when $k=1$. Section \ref{demi-eigen} is divided into two subsections, in the first one we prove that the generalized principal eigenvalues bound the validity of the maximum and minimum principle; in the second subsection we describe some unusual phenomena.  Section \ref{existence} is dedicated to the existence of solutions for the Dirichlet problem and  existence of the principal eigenfunction. In the last section we prove that $C^2$ strictly convex domains are 
  ``hula hoop domains''.

\section{Barrier functions, bounds, Hopf lemma}\label{barriers}

For convenience of the reader, we begin this section by recalling the definition of viscosity solution and some facts concerning the operators $\Pmk$ and $\Ppk$. 

Let us denote by $\mathbb S^N$ the set of $N\times N$ real symmetric matrices, endowed with the standard partial order: $X\leq Y$ in $\mathbb S^N$ if $\left\langle X\xi,\xi\right\rangle\leq\left\langle Y\xi,\xi\right\rangle$ $\forall\xi\in\Rn$. The identity matrix will be denoted by $I$ and the trace of $X\in\mathbb S^N$ by ${\rm tr}(X)$. A continuous mapping $F:\Omega\times\R\times\Rn\times\mathbb S^N\mapsto\R$ is \emph{degenerate elliptic} if it is nondecreasing in the matrix argument:
for any $(x, r, \xi)\in\Omega\times\R\times\Rn$
\begin{equation}\label{degell}
F(x,r,\xi,X)\leq F(x,r,\xi,Y) \quad\text{whenever}\;X\leq Y.
\end{equation}
\begin{definition}\label{visc} $u$ is a viscosity supersolution of 
$$F(x,u,\grad u, D^2u)=0\ \mbox{in}\ \Omega$$
if it is lower semicontinuous in $\Omega$ and for any $x$ in $\Omega$, for any $C^2$ function $\varphi$
touching $u$ from below at $x$ then
$$F(x,u,\grad \varphi(x),D^2 \varphi(x))\leq 0.$$ 

Analogously, $u$ is a viscosity subsolution
if it is upper semicontinuous in $\Omega$ and for any $x$ in $\Omega$, for any $C^2$ function $\varphi$
touching $u$ from above at $x$ then
$$F(x,u,\grad \varphi(x),D^2 \varphi(x))\geq 0.$$ 

A continuous function $u$ is a viscosity solution if it is both a subsolution and a supersolution.
\end{definition}

If $X\leq Y$ in $\mathbb S^N$, the Courant's min-max representation formula for eigenvalues implies that $\lambda_i(X)\leq\lambda_i(Y)$, for $i=1,\ldots,N$. In particular the operators $\Pmk$ and $\Ppk$ satisfy \eqref{degell}. Moreover the representation formula
\begin{equation*}
\Pmk(X)=\min\left\{\sum_{i=1}^k\left\langle X\xi_i,\xi_i\right\rangle\,|\,\text{$\xi_i\in\Rn$ and $\left\langle\xi_i,\xi_j\right\rangle=\delta_{ij}$, for $i,j=1,\ldots,k$}\right\},
\end{equation*}
see \cite[Lemma 8.1]{CLN1}, allows us to obtain easily the inequalities
\begin{equation}\label{suba}
\Pmk(Y)\leq{\mathcal P}^\pm_k(X+Y)-{\mathcal P}^\pm_k(X)\leq\Ppk(Y)
\end{equation}
and deduce the superadditivity (subadditivity) property  of $\Pmk$ ($\Ppk$).

\bigskip

We will consider  a couple of radial barrier functions in the paper and hence we recall
the following elementary Lemma that can be found e.g. in \cite{CutLe}.
\begin{lemma} \label{elementary} Let $\eta\in C^2([0,\,b])$, with $0<b$ such that $\eta^\prime(0)=0$.
Set $v(x)=\eta(|x|)$ in $\ol B_b$. Then, $v$ is $C^2(\ol B_b)$ and, for $x\neq 0$,
the eigenvalues of $D^2v(x)$  are $\eta''(|x|)$ and $\eta'(|x|)/|x|$, 
and the  (algebraic) multiplicity of $\eta'(|x|)/|x|$ is equal to $N-1$,  
if $\eta^{\prime\prime}(|x|)\not=\eta^\prime(|x|)/|x|$, and $N$ otherwise.  For $x=0$, they are all equal to $\eta^{\prime\prime}(|x|)$.
\end{lemma} 

We start with a computation that leads to a remark on the Hopf lemma for the operator $\Pmk(D^2\cdot)+H(x,\grad\cdot)$.
In $B_R=B_R(0)$, the ball of radius $R$ and center the origin, 
let 
\begin{equation}\label{barrier}
w(x)=(R^2-|x|^2)^\gamma\quad \mbox{ with } \gamma>1.
\end{equation}
 By Lemma \ref{elementary} or a straightforward computation the eigenvalues of the Hessian of $w$ are
$$\lambda_i(D^2w)=-2\gamma(R^2-|x|^2)^{\gamma-1}<0\ \mbox{ for }\ i=1,\dots,N-1$$ while
\begin{eqnarray*}
\lambda_N(D^2w)&=&-2\gamma(R^2-|x|^2)^{\gamma-1}+4|x|^2\gamma(\gamma-1)(R^2-|x|^2)^{\gamma-2}\\
&=&
2\gamma(R^2-|x|^2)^{\gamma-2}((2(\gamma-1)+1)|x|^2-R^2).
\end{eqnarray*}

In this way, from \eqref{SC1}  
\begin{equation*} 
\begin{split}
\Pmk(D^2w)+H(x,\grad w)&\leq \Pmk(D^2w)+b\left|\grad w\right|\\
&=2\gamma(R^2-|x|^2)^{\gamma-1}\left(b|x|-k\right)\leq0\qquad\text{if \,$bR\leq k$,}
\end{split}
\end{equation*}
so that $w$ is a positive supersolution, for $k<N$,  of $\Pmk(D^2w)+H(x,\grad w)=0$ in $B_R$, which is zero on the boundary and such that the outer normal derivative $\partial_\nu w(x)=0$ for $x$ on $\partial B_R$. 
This proves the following remark.

\begin{remark}\label{remHopf} For any $k<N$, the Hopf lemma does not hold in general for supersolutions of 
$\Pmk(D^2\cdot)+H(x,\grad\cdot)$, i.e. there exists  a positive supersolution  in $B_R$ which is zero 
together with its gradient at the boundary.\end{remark}

 Moreover the extension
\begin{equation*}
\bar w(x)=\begin{cases}
w(x) & \text{if $|x|<R$}\\
0 & \text{otherwise}
\end{cases}
\end{equation*}
yields, for $\gamma>2$, a counterexample of $C^2$-function invalidating the strong minimum principle. 

In \cite{GV} the authors dealt with the removable singularities issue for second order elliptic operators whose principal 
part is a weighted version of ${\mathcal P}^\pm_k$.  By means of an explicit counterexample they deduced the 
sharpness of the condition $bR\leq k< N$ for the validity of the weak maximum/minimum principle in the cases 
$H(x,\grad u)=\pm b|\grad u|$. For the reader's convenience we report the proof in the case of the minimum principle.
Assume $bR\leq k$ and by contradiction let $v$ be a lower semicontinuous function such that
\begin{equation*}
\begin{cases}
\Pmk(D^2v)+H(x,\nabla v)\leq0 & \text{in $\Omega\subset B_R$}\\
\qquad\displaystyle\liminf_{x\to\partial\Omega}v(x)\geq0
\end{cases}
\end{equation*}
and $v(x_0)<0$ for some $x_0\in\Omega$.\\
Set $\varphi(x)=\varepsilon|x|^2$ and $0<\varepsilon<-\frac{v(x_0)}{R^2}$. Since
$$
\liminf_{x\to\partial\Omega}(v-\varphi)(x)\geq-\varepsilon R^2>v(x_0)\geq (v-\varphi)(x_0)
$$
then
$$
\inf_{x\in\Omega}(v-\varphi)(x)=(v-\varphi)(x_\varepsilon),\quad x_\varepsilon\in\Omega.
$$
Using $\varphi$ as test function we get
\begin{equation*}
\begin{split}
0&\geq \Pmk (D^2\varphi(x_\varepsilon))+H(x_\varepsilon,\nabla \varphi(x_\varepsilon))\\
&\geq 2\varepsilon k -2\varepsilon b|x_\varepsilon|\\
&>2\varepsilon(k-bR)
\end{split}
\end{equation*}
a contradiction. For the sharpness of the condition see Example \ref{mu=0}.

 Summarizing we can assert that for $H$ fulfilling (\ref{SC1})
\begin{proposition}
$\Pmk(D^2\cdot)+H(x,\grad\cdot)$
does not satisfy the strong minimum principle in any bounded domain $\Omega$. \\
On the other hand the weak minimum principle holds true in $\Omega\subseteq B_R$ if $bR\leq k$ 
and the condition $bR\leq k$ is sharp in the case $H(x,\grad u)=-b|\grad u|$.
\end{proposition}

For later purposes we need to compare the distance function  to the boundary of  $\Omega$ i.e. $\displaystyle d(x)=\inf_{y\in\partial\Omega}|y-x|$ with subsolutions of (\ref{basic_equation}). This is the content of the next propositions.  

\medskip
\begin{proposition}[{\bf Hopf for subsolutions}]\label{Hopf}
Let $\Omega$ be a bounded $C^2$-domain and let $u$ satisfy 
\begin{equation*}
\left\{\begin{array}{lc}
\Pmk(D^2u)+H(x,\grad u)\geq 0
 &\mbox{in}\ \Omega\\
u< 0 &\mbox{in}\ \Omega.
\end{array}
\right.
\end{equation*}
Then there exists a positive constant $C=C(\Omega,u,k,b)$ such that
$$u(x)\leq -C d(x).$$
\end{proposition}
\begin{proof}
The proof is quite standard. We report it for the sake of completeness.
The conditions on $\Omega$ imply the existence of a positive constant $\delta$, depending on $\Omega$, such that for any $x\in\Omega_\delta=\left\{x\in\Omega\,|\,d(x)<\delta\right\}$ there are a unique $y\in\partial\Omega$ for which $d(x)=|y-x|$ and a ball $B_{2\delta}(\overline y)\subset\Omega$ such that $\overline {B_{2\delta}(\overline y)}\cap\left(\Rn\backslash\Omega\right)=\left\{y\right\}$ (see \cite[Lemma 14.16]{GT} for details).\\
Let us fix an arbitrary $x_0\in\Omega_\delta$ and consider the smooth negative radial function
$$v(x)=\beta\left(e^{-2\alpha\delta}-e^{-\alpha|x-\overline{y_0}|}\right)$$
in the annular region $A=B_{2\delta}(\overline {y_0})\backslash B_{\delta}(\overline {y_0})$.
For $\alpha>\left(\frac{k-1}{\delta}+b\right)$ and $\displaystyle\beta=\frac{\sup_{\Omega\backslash \Omega_\delta}u}{\left(e^{-2\alpha\delta}-e^{-\alpha\delta}\right)}$, a direct calculation (or Lemma \ref{elementary}) yields
\begin{equation*}
\begin{split}
\Pmk(D^2v(x))+H(x,\grad v(x))&\leq\Pmk(D^2v(x))+b\left|\grad v(x)\right|\\
&=\alpha\beta e^{-\alpha|x-\overline{y_0}|}\left(\frac{k-1}{|x-\overline{y_0}|}+b-\alpha\right)< 0\quad\text{in $A$}
\end{split}
\end{equation*}
and 
$$\limsup_{x\to\partial A}(u-v)(x)\leq0.$$
Using the comparison principle 
between a classical strict supersolution and a viscosity subsolution, we get
$$u(x_0)\leq v(x_0)=\beta\left(e^{-\alpha|y_0-\overline{y_0}|}-e^{-\alpha|x_0-\overline{y_0}|}\right)\leq-\alpha\beta e^{-2\alpha\delta}d(x_0).$$ Moreover since $\displaystyle\max_{\Omega\backslash\Omega_\delta}\frac{u(x)}{d(x)}<0$ we conclude by taking $C$ small enough.
\end{proof}

\begin{remark}\label{strong_MP}
\rm 
Standard procedures allow us to deduce from the above computation that the strong maximum principle holds for $\Pmk(D^2\cdot)+H(x,\grad\cdot) + \mu \cdot$ for any $\mu\in\R$.
\end{remark}

In Proposition~\ref{upbound}, we shall prove that for any $\gamma\in(0,1)$ and any subsolution 
$u$ of $\Pmo(D^2u)+H(x,\grad u)=f(x)$ in $\Omega$,
the ratio $\frac{u(x)}{d(x)^\gamma}$ is bounded from above by a constant $C$,  without requiring further assumptions on $\Omega$. The constant $C$ depends in particular on $\gamma$ and blows up for $\gamma\to1$. In order to obtain a similar bound with $\gamma=1$ and in the general case of subsolutions of the equation \eqref{basic_equation}, we restrict to convex domains $\Omega$ satisfying the following assumption: there exist $R>0$ and $Y\subset\R^N$, depending on $\Omega$, such that
\begin{equation}\label{strict_convexity}
\Omega=\bigcap_{y\in Y}B_R(y).
\end{equation}
For any $R>0$ we define the class $\mathcal C_R$ of such domains, i.e. $${\mathcal C}_R:=\left\{\text{$\Omega\subset\R^N$ :  representation formula \eqref{strict_convexity} holds}\right\}, \ \  \text{ and set } \  \mathcal{C} 
=\bigcup_{R>0}\mathcal{C}_R.$$
The class $\cC$ includes the set of bounded domains with $C^2$-boundary which are strictly convex in the sense that all the principal curvatures of the surface $\partial\Omega$ are 
positive everywhere. Indeed, we shall give, in section \ref{convex}, the proof of the following  
\def\ul{\underline}\def\gk{\kappa}
\begin{proposition}\label{hulahoop} Let $\gO$ be a bounded domain with $C^2$-boundary. 
Let $\kappa_i(x)$ denote the principal curvatures of $\partial\Omega$ at $x$ for $i=1,\ldots,N-1$, set 
\[\ul{\gk}=\min\{\gk_i(x)\,:\, i=1,\ldots,N-1,\ x\in\pl\gO\},\] 
and assume that $\ul{\gk}>0$. If $\,R\geq 1/\ul{\gk}$, then 
$\,\gO\in\cC_R$. 
\end{proposition}
By means of \eqref{strict_convexity} we show that the distance function $d(x)$ is an upper barrier for any 
subsolutions of \eqref{basic_equation}.

\begin{proposition}\label{upbound}
Let $m$ be a positive constant and let $u$ satisfy 
$$\left\{\begin{array}{lc}
\Pmo(D^2u)+H(x,\grad u)\geq -m
 &\mbox{in}\ \Omega\\
u\leq 0 &\mbox{on}\ \partial\Omega.
\end{array}
\right.
$$
Then for any $\gamma\in(0,1)$ there exists $C=C(\gamma, b,  m, 	\left\|u^+\right\|_\infty)$ such that 
 $$u(x)\leq C d(x)^\gamma.$$
Let $R>0$, $\Omega\in{\mathcal C}_R$ and $u$ be a solution of
$$\left\{\begin{array}{lc}
\Ppk(D^2u)+H(x,\grad u)\geq -m
 &\mbox{in}\ \Omega\\
u\leq 0 &\mbox{on}\ \partial\Omega.
\end{array}
\right.
$$
If $H$ satisfies \eqref{SC1} and  $bR<k$ then there exists $C=C(\Omega, b, k, m)$ such that 
\begin{equation}\label{barrier1}
u(x)\leq C d(x).
\end{equation}
\end{proposition}
\begin{proof} 
Let $\Omega_\delta=\left\{x\in\Omega\,|\,d(x)<\delta\right\}$ with
\begin{equation}\label{delta}
\delta=\min\left(\frac{1-\gamma}{2b},\left(\frac{\gamma(1-\gamma)}{4m}\left\|u^+\right\|_\infty\right)^{\frac{1}{2}}\right) 
\end{equation}
and without loss of generality we may assume $u^+\not\equiv 0$. For $x_0\in\Omega_\delta$, take $y_0\in\partial\Omega$ such that $d(x_0)=|x_0-y_0|$ and consider the function 
$v(x)=C|x-y_0|^\gamma$, where $C=\frac{\left\|u^+\right\|_\infty}{\delta^\gamma}$. Then $v(x)$ satisfies in $B_\delta(y_0)\cap\Omega$ 
\begin{equation*}
\begin{split}
\Pmo(D^2v(x))+H(x,\grad v(x))&\leq C\gamma|x-y_0|^{\gamma-2}\left(\gamma-1+b|x-y_0|\right)\\
&\leq -C\gamma\frac{1-\gamma}{2}\delta^{\gamma-2}<-m.
\end{split}
\end{equation*}
Moreover
$$
u(x)\leq v(x)\quad\text{for any $x\in\partial(B_\delta(y_0)\cap\Omega)$}
$$
and by comparison $u(x_0)\leq v(x_0)=Cd(x_0)^\gamma$. Since $x_0$ is arbitrary we obtain the desired inequality $u(x)\leq Cd(x)^\gamma$ in $\Omega_\delta$ and the same conclusion is still true in $\Omega\backslash\Omega_\delta$ by the choice of the constant $C$.

\medskip

For the second inequality, fix any $y\in Y$ and consider the function $v_y(x)=M(R^2-|x-y|^2)$, where $M=\frac{m}{(k-bR)}$. Note that $v_y(x)\geq 0$ for all $x\in\overline B_R(y)$ and hence $v_y(x)\geq0$ in $\overline\Omega$. Then
\begin{equation*}
\begin{split}
\Ppk(D^2v_y(x))+H(x,\grad v_y(x))&\leq 2M(-k+b|x-y|)\\
&\leq-2M(k-bR)<-m\quad\text{in $B_R(y)$}
\end{split}
\end{equation*}
and by comparison 
\begin{equation}\label{comparison}
u(x)\leq v_y(x)\quad\text{in $\overline\Omega$}.
\end{equation}
This shows that
$$
u(x)\leq Cd(x)\quad\text{for all $x\in\overline\Omega$}
$$
with $C=2MR$. To see this, let $x\in\Omega$ and select $z\in\partial\Omega$ so that $d(x)=|x-z|$. Then select $y\in Y$ so that $z\notin B_R(y)$. Since $x\in B_R(y)$, we have
\begin{equation*}
\begin{split}
R^2-|x-y|^2&=(R-|x-y|)(R+|x-y|)\leq 2R(R-|x-y|)\\
&=2Rd(x,\partial B_R(y))= 2R|x-z|=2Rd(x)
\end{split}
\end{equation*}
and we conclude by \eqref{comparison}. 
\end{proof}
 
We conclude this section by observing that the upper bound \eqref{barrier1} fails to be true if the boundary $\partial\Omega$ is flat, at least if $\Omega$ is unbounded. Indeed in the case of the halfspace $$\Omega=\left\{x=(x_1,\ldots,x_N)\in\Rn\,:\,x_1>0\right\},$$
the function $u(x)=x_1^\gamma$ is a solution in $\Omega$ of $\Ppk(D^2u)=0$ for any $\gamma\in(0,1)$ and $k<N$, but on the other hand
the ratio $\frac{u(x)}{d(x)}=\frac{1}{x_1^{1-\gamma}}$ is unbounded near $x_1=0$.

\section{Lipschitz regularity, compactness}\label{compactness}
In this section we will study the Lipschitz regularity of viscosity solutions of 
\begin{equation}\label{holder1}
\left\{\begin{array}{lc}
\Pmo(D^2u)+H(x,\grad u)= f(x)& \mbox{in}\ \Omega\\
u=0& \mbox{on}\ \partial\Omega
\end{array}
\right.
\end{equation}
and, in a dual fashion,  of
\begin{equation}\label{holder2}
\left\{\begin{array}{lc}
\Ppo(D^2u)+H(x,\grad u)= f(x)& \mbox{in}\ \Omega\\
u=0& \mbox{on}\ \partial\Omega,
\end{array}
\right.
\end{equation}
where $f$ is  continuous and bounded in $\Omega$.  

\begin{proposition}\label{Lipschitz_regularity}  
Let $\Omega\in{\mathcal C}_R$. If $H$ satisfies \eqref{SC1} and $bR<1$, then the solutions $u$ of \eqref{holder1} and \eqref{holder2} are Lipschitz continuous in $\overline\Omega$. The Lipschitz  norm of $u$ can be bounded  by a constant depending
 only on $\Omega$, $b$ and the $L^\infty$ norms of $u$ and $f$.
\end{proposition} 
\begin{proof}
We shall write the proof in the case $\Pmo$, since if $v$ is a solutions of \eqref{holder2}, then $u=-v$ is a solution of $\Pmo(D^2u)+\tilde H(x,\grad u)=-f(x)$ in $\Omega$, where $\tilde H(x,\xi)=-H(x,-\xi)$ satisfies in turn $\eqref{SC1}$.\\
Let $u$ be a solution of \eqref{holder1}. It is sufficient to show that for any $x$, $y\in\overline\Omega$ such that $|x-y|<\delta$, where $\delta$ is a positive constant to be determined, then
$$
u(x)-u(y)\leq L|x-y|
$$
with $L=L(\Omega, b, \left\|u\right\|_\infty, \left\|f\right\|_\infty)$.\\
Fix $\theta\in(1,2)$ and consider
$$
v(x)=|x|-|x|^\theta,\quad x\in B_1.
$$
The function $v$ is strictly positive for $x\neq0$ and satisfies the inequality
\begin{equation}\label{eq1_Lipschitz}
\Pmo(D^2v(x))+H(x,\grad v(x))\leq-\theta(\theta-1)|x|^{\theta-2}+b(1+\theta|x|^{\theta-1}),\quad x\in B_1\backslash\left\{0\right\}.
\end{equation}
Since the right hand side in \eqref{eq1_Lipschitz}  tends to $-\infty$ as $|x|\to0$, we can then pick a  $\delta=\delta(b,\theta,\left\|f\right\|_\infty)\in (0,\,1)$ 
such that
$$
\Pmo(D^2v(x))+H(x,\grad v(x))< -\left\|f\right\|_\infty \quad\text{in $\in B_\delta\backslash\left\{0\right\}$}.
$$
Moreover, in view of Proposition~\ref{upbound}, there exists a positive constant $C=C(\Omega, b, \left\|f\right\|_\infty)$ such that
\begin{equation}\label{eq2_Lipschitz}
-u(x)\leq Cd(x)\quad\forall x\in\overline\Omega.
\end{equation}
For $x_0$, $y_0\in\overline\Omega$, with $|x_0-y_0|<\delta$ and $L=\max\left(\frac{2\left\|u\right\|_\infty}{\delta-\delta^\theta},\frac{C}{1-\delta^{\theta-1}}\right)$, let
\begin{equation}\label{function v_0}
v_{y_0}(x):=u(y_0)+Lv(x-y_0),\quad x\in B_\delta(y_0).
\end{equation}
By construction 
$$
\Pmo(D^2v_{y_0}(x))+H(x,\grad v_{y_0}(x))< -\left\|f\right\|_\infty\quad\text{in $B_\delta(y_0)\backslash\left\{y_0\right\}$}
$$
and
$$
v_{y_0}(y_0)=u(y_0).
$$
We claim that 
\begin{equation}\label{eq3_Lipschitz}
u(x)\leq v_{y_0}(x)\quad\text{on $\partial(B_\delta(y_0)\cap\Omega)$},
\end{equation}
so that the comparison principle yields the conclusion
$$
u(x_0)\leq v_{y_0}(x_0)\leq u(y_0)+L|x_0-y_0|.
$$
To prove the inequality \eqref{eq3_Lipschitz} we note that for any $x\in\partial B_\delta(y_0)\cap\Omega$
$$
v_{y_0}(x)=u(y_0)+L(\delta-\delta^\theta)\geq u(y_0)+2\left\|u\right\|_\infty\geq u(x),
$$
while if $x\in \overline B_\delta(y_0)\cap\partial\Omega$,  we obtain in view of \eqref{eq2_Lipschitz}, together with the choice of $L$,
\begin{equation*}
\begin{split}
u(x)=0&\leq u(y_0)+ C d(y_0)\leq u(y_0)+C |x-y_0|\\
&\leq u(y_0)+L(|x-y_0|-|x-y_0|^\theta)=v_{y_0}(x)
\end{split}
\end{equation*}
as we wanted to show.
\end{proof}

The conditions concerning  the geometry of $\Omega$ and the smallness of the Hamiltonian in the Proposition~\ref{Lipschitz_regularity}, i.e. 
\begin{equation}\label{cond_Omega}
\Omega\in{\mathcal C}_R\quad \text{and}\quad bR<1,
\end{equation} are only used to get the inequality \eqref{eq2_Lipschitz}, in order to apply comparison principle up to the boundary. For this reason and following the arguments of the previous proof, it is  easy to obtain interior Lipschitz regularity for any bounded domain $\Omega$ and any $H$ satisfying \eqref{SC1}, assuming $u$ to be merely a subsolution of \eqref{holder1}.

Moreover the assumptions \eqref{cond_Omega} can be dropped if we require that the subsolution $u$  satisfies\eqref{eq2_Lipschitz}. 
These observations are summarized as follows.

\begin{proposition}\label{Lip-regularity}
Suppose that $\Omega$ is a bounded domain and $H$ satisfies condition \eqref{SC1}. The following holds:
\begin{enumerate}
	\item[i)] any subsolution $u$ of \eqref{holder1} is a locally Lipschitz continuous function in $\Omega$;
	\item[ii)] any subsolution $u$ of \eqref{holder1} that satisfies \erf{eq2_Lipschitz} 
	for some constant $C$ 
	is  Lipschitz continuous in $\overline\Omega$.
\end{enumerate}
The Lipschitz norm of $u$ can be estimated by a constant which depends 
on $b$ and the $L^\infty$ norms of $u$ and $f$.\\
Finally the same conclusion holds for supersolutions $u$ of \eqref{holder2}, with 
\erf{eq2_Lipschitz} replaced by the inequality $u\leq Cd$ in $\ol\gO$.
\end{proposition}

This globally Lipschitz regularity result for nonnegative subsolutions of \eqref{holder1},
a consequence of Proposition~\erf{Lip-regularity} ii), 
is quite surprising, considering that the global $C^{0,\gamma}$-regularity may fails for any $\gamma\in(0,1]$ in the class of nonpositive subsolutions of \eqref{holder1}. Here below an example: the nonpositive radial function
$$
u(x)=\begin{cases}
\frac{1}{\log(1-\delta)} & \text{if $|x|\leq\delta$}\\
\frac{1}{\log(1-|x|)} & \text{if $\delta<|x|<1$}\\
0 & \text{if $|x|=1$},
\end{cases}
$$ 
is convex for $\delta\in (0,\,1)$ close to 1 and  $$\Pmo(D^2u(x))\geq0\quad\text{in $B_1$}.$$
On the other hand for any $\gamma\in(0,1]$
$$\sup_{\substack{x,y\in \overline B_1\\x\neq y}}\frac{|u(x)-u(y)|}{|x-y|^\gamma}=+\infty.$$

\section{Demi-eigenvalues}\label{demi-eigen}
\subsection{Maximum and minimum principle}
We now investigate the relationship between the generalized principal eigenvalues  $\overline{\mu}_k^\pm$ and ${\mu}_k^\pm$ given in the introduction and the validity of the maximum and minimum principle.

In the following we shall sometimes need to reinforce the assumptions on the Hamiltonian $H$. In particular: 
\begin{equation}\label{SC2}
H(x,t\xi)=tH(x,\xi)\quad\forall(x,t,\xi)\in\Omega\times\R_+\times\Rn,
\tag{SC 2}
\end{equation}
\begin{equation}\label{SC3}
\exists\,\omega\,\, \text{modulus of continuity s.t. }\,
\left|H(x,\xi)-H(y,\xi)\right|\leq\omega\left(|x-y|\left(1+|\xi|\right)\right).
\tag{SC 3}
\end{equation}

Observe that \eqref{SC2} implies \eqref{SC1} with $b=\sup_{(x,\xi)\in \Omega\times B_1}|H(x,\xi)|$ hence this will be the meaning of  $b$ under condition  \eqref{SC2}. Furthermore \eqref{SC2} and \eqref{SC3} imply that $H$ is Lipschitz continuous in the following sense:
\[
|H(x,\xi)-H(y,\xi)|\leq C|x-y||\xi|
\]
for some constant $C>0$.
Indeed, for $\eta=\frac{\xi}{|\xi||x-y|}$, $$|H(x,\xi)-H(y,\xi)|=|H(x,\eta)-H(y,\eta)||\xi||x-y|$$ and
$$ |H(x,\xi)-H(y,\xi)|\leq\omega(|x-y|(1+|\eta|))|\xi||x-y|\leq \omega(1+\mathrm{diam}(\Omega))|x-y||\xi|.$$

\begin{theorem}\label{MaxPrinc} Let $\Omega$ be a bounded domain.
Under the assumption \eqref{SC2}-\eqref{SC3},  the operator 
$$\Pmk(D^2\cdot)+H(x,\grad \cdot)+\mu\cdot$$ satisfies

\begin{itemize}

\item[i)] the minimum principle in $\Omega$ for $\mu<\overline{\mu}_k^-$,
\item[ii)]  the maximum principle in $\Omega$ for $\mu<\overline{\mu}_k^+$.
\end{itemize}
\end{theorem}

\begin{proof} The proof follows the argument of \cite{BD}.

Without loss of generality we can suppose that $\mu\geq 0$, because otherwise the results are well known. 
We shall detail the case i) of the minimum principle,  since with minor changes the arguments prove ii) as well.  We argue by contradiction by assuming that $v$ is a solution of
\begin{equation}\label{eq1prop}
\left\{\begin{array}{lc}
\Pmk(D^2v)+H(x,\grad v)+\mu v\leq 0& \mbox{in}\ \Omega\\
\displaystyle \liminf_{x\to\partial\Omega}v(x)\geq0& 
\end{array}
\right.
\end{equation}
and $v(x_0)<0$ for some $x_0\in\Omega$. 

By the definition of $\overline{\mu}_k^-$ there exists $\rho\in(\mu,\overline{\mu}_k^-)$ and  $u<0$ in $\overline\Omega$, a solution of
\begin{equation}\label{eq2prop}
\Pmk(D^2u)+H(x,\grad u)+\rho u\geq0\quad\text{in $\Omega$}.
\end{equation}
The function $\frac{v}{u}$ is upper semicontinuous in the compact set 
$$K=\left\{x\in\Omega\,:\,\frac{v(x)}{u(x)}\geq\frac{v(x_0)}{u(x_0)}\right\}$$
and if  $\gamma:=\sup_{x\in\Omega}\frac{v(x)}{u(x)}$, then
\begin{equation}\label{eq3prop}
\gamma=\sup_{x\in K}\frac{v(x)}{u(x)}<+\infty \quad \mbox{and}\quad 0<\frac{v(x_0)}{u(x_0)}\leq\gamma.
\end{equation}

For $0<\varepsilon<\gamma$ the lower semicontinuous function $v
-(\gamma-\varepsilon)u$ reaches its negative minimum in $\Omega$, say
$$
\min_{x\in\Omega}\left(v(x)-(\gamma-\varepsilon)u(x)\right)=v(x_\varepsilon)-(\gamma-\varepsilon)u(x_\varepsilon),\quad x_\varepsilon\in\Omega,
$$
since 
$$
\liminf_{x\to\partial\Omega}\left(v(x)-(\gamma-\varepsilon)u(x)\right)\geq\liminf_{x\to\partial\Omega}v(x)\geq0
$$
and by definition of the supremum there exists $y_\varepsilon\in\Omega$ such that 
$$
v(y_\varepsilon)-(\gamma-\varepsilon)u(y_\varepsilon)<0.
$$
Moreover, by lower semicontinuity, we can find a subdomain $\Omega'\subset\subset\Omega$, depending on $\varepsilon$ and containing $x_\varepsilon$, for which
\begin{equation}\label{eq4prop}
\min_{\partial\Omega'}\left(v(x)-(\gamma-\varepsilon)u(x)\right)>v(x_\varepsilon)-(\gamma-\varepsilon)u(x_\varepsilon)
\end{equation}
and a sequence $(x_k,y_k)\in\overline{\Omega'}\times\overline{\Omega'}$ such that
$$
v(x_k)-(\gamma-\varepsilon)u(y_k)+\frac k2|x_k-y_k|^2=\min_{(x,y)\in\overline{\Omega'}\times\overline{\Omega'}}\left(v(x)-(\gamma-\varepsilon)u(y)+\frac k2|x-y|^2\right).
$$
Using \cite[Lemma 3.1]{CIL}, up to subsequences, we have 
$$\frac k2|x_k-y_k|^2\to0,\;\; (x_k,y_k)\to(\hat {x}_\varepsilon,\hat {x}_\varepsilon)\,\,\,\text{for some}\,\, \hat {x}_\varepsilon\in\Omega'$$ and 
$$(v(x_k),(\gamma-\varepsilon)u(y_k))\to(v(\hat {x}_\varepsilon),(\gamma-\varepsilon)u(\hat {x}_\varepsilon))\quad \text{for $k\to+\infty$.}$$ 
Hence $(x_k,y_k)\in\Omega'\times\Omega'$ for large $k$   and in view of \cite[Theorem 3.2]{CIL} there exist $X_k$ and $Y_k$, $N\times N$ symmetric matrices, such that
$$
X_k\geq Y_k,\;\left(k(y_k-x_k),X_k\right)\in \overline J^{2,-}v(x_k),\;\left(k(y_k-x_k),Y_k\right)\in \overline J^{2,+}(\gamma-\varepsilon)u(y_k).
$$   
Since the function $(\gamma-\varepsilon)u(x)$ is still a solution of (\ref{eq2prop}) by the homogeneity assumption (\ref{SC2}), we have from {(\ref{eq1prop})-(\ref{eq2prop})-(\ref{SC3})}, that
\begin{equation*}
\begin{split}
\mu v(x_k)&\leq-\Pmk(X_k)-H(x_k,k(y_k-x_k))\\
&\leq-\Pmk(Y_k)-H(y_k,k(y_k-x_k))+\omega\left(|x_k-y_k|(1+k|x_k-y_k|)\right)\\
&\leq	\rho(\gamma-\varepsilon)u(y_k)+\omega\left(|x_k-y_k|(1+k|x_k-y_k|)\right).
\end{split}
\end{equation*}
Sending $k\to+\infty$
\begin{equation}\label{eq5prop}
\mu v(\hat{x}_\varepsilon)\leq\rho(\gamma-\varepsilon)u(\hat{x}_\varepsilon).
\end{equation}
If $\mu=0$ this is a contradiction.  Otherwise, for $\mu>0$,
since $\gamma u(\hat{x}_\varepsilon)\leq v(\hat{x}_\varepsilon)$ we deduce from (\ref{eq5prop})
that
$$1<\frac\rho\mu\leq\frac{\gamma}{\gamma-\varepsilon};$$
which is a contradiction for small $\varepsilon$.
\end{proof}

The same proof as above works for general, positively homogeneous of degree one, 
degenerate elliptic operators 
$F(x,\grad \cdot, D^2\cdot)$, to which the proof of comparison principle 
applies (see \cite[Theorem 3.3]{CIL}). 

\medskip
Theorem \ref{MaxPrinc} implies the following
\begin{corollary}\label{corbound}
Under the assumption \eqref{SC2}-\eqref{SC3}, if $B_{R_1}\subset\Omega$,  then
\begin{equation}\label{bounds1-eigenvalue}
\overline{\mu}_k^-\leq\frac{2(k+bR_1)(2+k+bR_1)}{R_1^2}.
\end{equation}
Moreover if $\Omega\subset B_{R_2}$ and $bR_2\leq k$, then
\begin{equation}\label{bounds2-eigenvalue}
\overline{\mu}_k^-\geq\frac{2(k-bR_2)}{R_2^2}.
\end{equation}
\end{corollary} 

\begin{proof} For $B_{R_1}\subset\Omega$ consider the function
$$w(x)=-(R_1^2-|x|^2)^2$$
extended to zero outside of $B_{R_1}$, as in \cite{BNV}. 
Then  
\begin{equation*}
\begin{split}
\sup_{|x|<R_1}\frac{\Pmk(D^2w)+H(x,\grad w)}{-w}&\leq \sup_{|x|<R_1}\frac{\Pmk(D^2w)+b|\grad w |}{-w}\\
&\leq 4\sup_{|x|<R_1}\left(\frac{k+bR_1}{(R_1^2-|x|^2)}-\frac{2|x|^2}{(R_1^2-|x|^2)^2}\right). 
\end{split}
\end{equation*}
In the set $\Omega_1=\left\{x\in B_{R_1}:\,|x|^2\geq\frac{R_1^2(k+bR_1)}{2+k+bR_1}\right\}$ we have
$$\frac{k+bR_1}{(R_1^2-|x|^2)}-\frac{2|x|^2}{(R_1^2-|x|^2)^2}\leq0,$$
while in $\Omega_2=B_{R_1}\backslash\Omega_1$
$$\frac{k+bR_1}{(R_1^2-|x|^2)}-\frac{2|x|^2}{(R_1^2-|x|^2)^2}\leq\frac{k+bR_1}{(R_1^2-|x|^2)}\leq\frac{(k+bR_1)(2+k+bR_1)}{2R^2_1}.$$
Hence $v$ is a negative solution in $\Omega$ of 
$$\Pmk(D^2w)+H(x,\grad w)+\frac{2(k+bR_1)(2+k+bR_1)}{R_1^2}w\leq 0,$$
which is zero on the boundary  $\partial\Omega$. This contradicts the minimum principle and,  by Theorem \ref{MaxPrinc},
$$\overline\mu^-_k\leq \frac{2(k+bR_1)(2+k+bR_1)}{R_1^2},$$
leading to (\ref{bounds1-eigenvalue}).\\
\smallskip

Let $\Omega\subset B_{R_2}$ and $w(x)=-(R_2^2-|x|^2)$. For $bR_2<k$ (the case $bR_2=k$ is trivial) we may assume as in the proof of Proposition~\ref{infinito} that $\overline\Omega\subset B_{R_2}$, so $w<0$ in $\overline\Omega$ and
\begin{equation*}
\begin{split}
\Pmk(D^2w)+H(x,\grad w)+\mu w&\geq\Pmk(D^2w)-b|\grad w|+\mu w\\
&=2(k-b|x|)+\mu\left(|x|^2-R_2^2\right)\\
&\geq 2(k-bR_2)-\mu R_2^2=0
\end{split}
\end{equation*}
if $\mu=\frac{2(k-bR_2)}{R_2^2}$ and therefore
$$\overline\mu_k^-\geq\frac{2(k-bR_2)}{R_2^2}.$$
\end{proof}

We now impose some conditions on the domain $\Omega$. For the maximum principle we get
\begin{proposition} \label{infinito} Under the assumption \eqref{SC1}, if $\Omega\subset B_R$ then, for any $k<N$, 
\begin{equation}\label{eqinfinito}
bR<k\quad\Longrightarrow\quad\mu_k^+=\overline{\mu}_k^+=+\infty.
\end{equation}
In particular, in the case $H\equiv0$, for any bounded domain $\Omega$,  $\mu_k^+=\overline{\mu}_k^+=+\infty$ and 
 the operator $\Pmk(D^2\cdot)+\mu\cdot$ satisfies the maximum principle for any $\mu$.
\end{proposition}

\begin{proof}
Choose any $\mu>0$ and assume without loss of generality that $\overline \Omega\subset B_R$ and $\gamma:=\frac{\mu R^2}{2(k-bR)}>1$, replacing if necessary $R$ with $R'>R$ in order that $k-bR'$ is positive and sufficiently close to 0.  Let $w$ be the function introduced in section \ref{barriers}, then $w(x)>0$ in $\overline\Omega$ and
\begin{equation*}
\begin{split}
&\;\quad\Pmk(D^2w(x))+H(x,\grad w(x))+\mu w(x)\\&\leq\Pmk(D^2w(x))+b|\grad w(x)|+\mu w(x)\\
&=-2\gamma k (R^2-|x|^2)^{\gamma-1}+2\gamma b|x|(R^2-|x|^2)^{\gamma-1}+\mu (R^2-|x|^2)^{\gamma}\\
&\leq (R^2-|x|^2)^{\gamma-1}(-2\gamma (k-bR) +\mu R^2)\\
&=0.
\end{split}
\end{equation*} 
By definition, we have obtained that 
$\mu_k^+\geq\overline{\mu}_k^+=+\infty$.
\end{proof}

For the minimum principle, the assumptions are slightly stronger.
\begin{theorem}\label{thm3} Let $\Omega\in\mathcal{C}_R$, and assume 
\eqref{SC2}-\eqref{SC3} and that $bR<k$. Then,  
$$\mu_k^- =\overline{\mu}_k^-,$$
and the minimum principle holds true if and only if $\mu<\mu_k^-$.
\end{theorem}
In order to prove Theorem \ref{thm3} we shall need the following proposition which proves that if $\Omega$ is a hula hoop domain, the bound $\overline\mu_k^-$ of Theorem~\ref{MaxPrinc} is sharp. We indeed exhibit a supersolution $v$ at level $\overline\mu_k^-$ which will invalidate the minimum principle. 
The result has been inspired by \cite[Proposition 3.2]{BCPR}.

\begin{proposition}\label{supersol}  
Assume \eqref{SC2}-\eqref{SC3}. Then $\overline{\mu}_k^-$ is 
finite and, if $\Omega\in\mathcal C_R$ and $bR<k$, there exists a nonpositive supersolution $v\not\equiv0$ of
\begin{equation*}
\left\{\begin{array}{lc}
\Pmk(D^2v)+H(x,\grad v)+\overline\mu_k^- v=0 & \text{in }\ \Omega\\
v=0 & \text{on }\ \partial \Omega.
\end{array}
\right.
\end{equation*}
\end{proposition}

For the proof of the proposition above, we need the following existence result that will be used also in the next section.

\begin{proposition} \label{boundedsub}   
Assume \erf{SC2} -\erf{SC3}. Let $\Omega\in\cC_R$ and $\mu<\ol\mu_k^-$, 
and assume that $bR<k$. Then, for $f$ bounded, there exist a subsolution $v$ 
and a supersolution $w$ of 
\beq\label{boundedsub0}
\cP_k^-(D^2u)+H(x,\grad u)+\mu u=f(x) \ \ \ \text{ in }\gO
\eeq
that satisfy 
$w\leq v$ in $\ol\Omega$ and $w=v=0$ on $\pl\gO$. 
\end{proposition} 
\begin{proof}[Proof of Proposition~\ref{boundedsub}.] Fix $\rho\in(\mu,\,\ol\mu_k^-)$, and, 
in view of the definition of $\ol\mu_k^-$, 
we may select a  real valued subsolution $\psi$ of 
\[
\cP_k^-(D^2\psi)+H(x,\grad \psi)+\rho \psi=0 \ \ \ \text{ in }\gO
\] 
such that $\psi<0$ in $\ol\gO$. We may assume by 
multiplying $\psi$ by a positive constant if necessary that 
$(\rho-\mu)\psi\leq -\|f\|_\infty$ in $\ol\gO$. 
It is now clear that $\psi$  is a subsolution 
of \erf{boundedsub0} or more precisely
\[
\cP_k^-(D^2\psi)+H(x,\grad \psi)+\mu \psi= \|f\|_\infty\ \ \ \text{ in }\gO.
\]
By translation, we may assume that $0\in\gO$. 
Since $\gO$ is a bounded, open, convex set, for any $\ep>0$, 
there is $\gd>0$ such that 
\[
(1+\ep)\gO\supset \gO^{\gd}:=\{x\in\R^N\,:\, \dist(x,\gO)<\gd\}. 
\]
We select such a $\gd=\gd(\ep)$ so that $0<\gd<\ep$. 

Define $\psi_\ep(x)=\psi((1+\ep)^{-1}x)$ for $x\in (1+\ep)\gO$ and note that $\psi_\ep$
is a subsolution of 
\[
(1+\ep)^{2}\cP_k^-(D^2\psi_\ep(x))+(1+\ep)H((1+\ep)^{-1}x,\grad \psi_\ep(x))
+\mu \psi_\ep(x)=\|f\|_\infty \ \ \ \text{ in }(1+\ep)\gO. 
\]
Thus, setting $H_\ep(x,\xi)=(1+\ep)^{-1}H((1+\ep)^{-1}x,\xi)$ and $\mu_\ep=
(1+\ep)^{-2}\mu$ , we see that $\psi_\ep$ is a subsolution of 
\[
\cP_k^-(D^2\psi_\ep)+H_\ep(x,\grad \psi_\ep)+\mu_\ep \psi_\ep=(1+\ep)^{-2}\|f\|_\infty \ \ \text{ in }\gO^\gd.
\]

For each $z\in B_\gd$, we define functions $\psi_\ep^z$ in $\ol\gO$ 
and $\widetilde H_{\ep}$ in $\ol\gO\tim\R^N$, respectively,  
by 
\[\psi_\ep^z(x)=\psi_\ep(x+z), \quad\text{ and }\quad \widetilde H_{\ep}(x,\xi)=\sup_{z\in B_\gd}H_\ep(x+z,\xi),
\]
and note that $\psi_\ep^z$ is a subsolution of 
\beq\label{boundedsub1}
\cP_k^-(D^2\psi_\ep^z)+\widetilde H_{\ep}(x,\grad \psi_\ep^z)+\mu_\ep \psi_\ep^z=(1+\ep)^{-2}\|f\|_\infty \ \ 
\text{ in }\gO.
\eeq

Set 
\[
W_\ep(x):=\max_{z\in \ol B_{\gd/2}}\psi_\ep^z(x)=\max_{y\in \ol B_{\gd/2}(x)}\psi_\ep(y)=
\max_{y\in \ol B_{\gd/2}(x)}\psi((1+\ep)^{-1}y) \ \ \text{ for }x\in\ol\gO,
\]
and observe that $W_\ep$ is upper semicontinuous in $\ol\gO$ and it is 
a subsolution of \erf{boundedsub1}, that $W_\ep\leq \max_{\ol\gO}\psi<0$ in $\ol\gO$,
and that the function 
\[
\widetilde H_\ep(x,\xi)=\sup_{z\in B_\gd}(1+\ep)^{-1}H((1+\ep)^{-1}(x+z),\xi)
\]  
satisfies \eqref{SC2} and \eqref{SC3}, with constant $(1+\ep)^{-1}b$ in place of $b$. 

Fix any $\ep>0$. We show that $W_\ep$ is bounded from below in $\ol\gO$. 
For this, we argue by contradiction and thus suppose that there 
is a sequence $(x_n)_{n\in\N}\subset\ol\gO$ such that $W_\ep(x_n)<-n$ 
for all $n\in\N$. We may assume up to extracting a subsequence that $(x_n)$ converges 
to some $x_0\in\ol\gO$. Moreover, we may assume that $x_n\in \ol B_{\gd/2}(x_0)$ 
for all $n$, which implies that, for any $n\in\N$, 
$x_0\in \ol B_{\gd/2}(x_n)$ and
\[
\psi_\ep(x_0)\leq W_\ep(x_n), 
\] 
which gives a lower bound of the sequence $(W_\ep(x_n))$, a contradiction.  

Next, we choose a sequence $(\ep_n)_{n\in\N}$ of positive numbers 
converging to zero, and,
for $n\in\N$, set $V_n=W_{\ep_n}$, $H_n=\widetilde H_{\ep_n}$, $\mu_n=\mu_{\ep_n}$,
and observe that, as $n\to +\infty$, $H_n \to H$ in $C(\ol\gO\tim\R^N)$, 
$\mu_n\to\mu$.  

Fix any $n\in\N$, and let $f_n(x)=(1+\ep_n)^{-2}f(x)$.
The standard construction of barrier functions for elliptic PDE yields 
a supersolution $W\in C(\ol\gO)$ of \erf{dir1} that satisfies $W=0$ on $\pl\gO$ 
and $W\geq 0$ in $\ol\gO$.  If $f\geq 0$  then just take $W\equiv 0$.
\beq\label{boundedsub2}
\cP_k^-(D^2u)+H_n(x,\grad u)+\mu_n u=f_n(x) \ \ \text{ in }\gO,
\eeq 
and $V_n$ is a  subsolution. We define the function $z_n$ in $\ol\gO$ by 
\[
z_n(x)=\inf\{u(x)\,:\, u
\text{ supersolution of \erf{boundedsub2}},\ V_n\leq u\leq W \text{ in }\ol\gO,\ u=0 
\ \text{ on }\,\pl\gO\}.
\]
By Perron procedure, the function $z_n$ is a ``viscosity solution'' of \erf{boundedsub2} in the sense 
that the upper semicontinuous envelope $(z_n)^*$ of $z_n$, given by
\[
(z_n)^*(x)=\inf_{r>0}\sup\{z_n(y)\,:\, y\in\ol\gO,\ |y-x|<r\},
\]
is a subsolution of \erf{boundedsub2} and the lower semicontinuous envelope 
$(z_n)_*$ of $z_n$, given by 
\[
(z_n)_*(x)=\sup_{r>0}\inf\{z_n(y)\,:\, y\in\ol\gO,\ |y-x|<r\},
\] 
is a supersolution of \erf{boundedsub2}.
It is clear that $\inf_{\ol\gO} V_n\leq (z_n)_*\leq (z_n)^*\leq  W$ in $\ol\gO$.  
If $u$ is a supersolution of \erf{boundedsub2} and if $V_n\leq u\leq  W$ in $\ol\gO$ and
$u=0$ on $\pl\gO$, then $u$ is supersolution of 
\[
\cP_k^-(D^2u)+H_n(x,\grad u)=f_n(x)-|\mu_n|\inf_{\ol\gO}V_n \ \ \text{ in }\gO.
\]
Proposition~\ref{upbound}, applied to $-u$, yields an inequality 
$u(x)\geq -C_nd(x)$  for all $x\in\ol\gO$ and some $C_n>0$, where 
$C_n$ is independent of the choice of $u$. This implies that $-C_n d\leq (z_n)_*
\leq (z_n)^*\leq  W$ in $\ol\gO$, which, in particular,  
ensures that $(z_n)_*=(z_n)^*=0$ on $\pl\gO$.

Now, we intend to send $n\to+\infty$. We claim that $\sup\|(z_n)_*\|_\infty<+\infty$. 
To check this, we argue by contradiction and suppose 
that $\sup\|(z_n)_*\|_\infty=+\infty$. 
We may assume up to a subsequence that $\lim_{n\to+\infty}\|(z_n)_*\|_\infty=+\infty$.
Set 
\[
Z_n(x)=\fr{(z_n)_*(x)}{\|(z_n)_*\|_\infty} \ \ \text{ for }x\in\ol\gO,\ n\in\N,
\]
and note that if we set
\[
M_0=\sup_{n\in\N}\fr{\|f\|_\infty}{\|(z_n)_*\|_\infty}+|\mu|, 
\]then
$Z_n$ is a supersolution of 
\[
\cP_k^-(D^2Z_n)+H_n(x,\grad Z_n)=M_0 \ \ \ \text{ in }\gO.
\]
Since $bk<R$, by applying Proposition~\ref{upbound} to $-Z_n$, we get, for some 
constant $M_1>0$, 
\beq\label{bc-0}
Z_n(x)\geq -M_1d(x) \ \ \ \text{ for all }x\in\ol\gO,\ n\in\N.  
\eeq

We take the lower relaxed limit of $(Z_n)_{n\in\N}$, that is, we set
\[
Z^-(x)=\liminf_{n\to+\infty}{\kern-2pt}_* Z_n(x)
=\sup_{r>0}\inf\{Z_n(y)\,:\, y\in\ol\gO,\,|y-x|<r,\, n>r^{-1}\}.
\]
It is a standard observation (see, e.g., \cite[Chapter 6]{CIL}) that $Z^-$ is lower semicontinuous in $\ol\gO$ 
and a supersolution of \erf{eqev}. It is clear that  
$Z^-\leq 0$ in $\ol\gO$ and $\min_{\ol\gO}Z^-=-1$. Moreover, it follows from
\erf{bc-0} that $Z^-=0$ on $\pl\gO$. According to Theorem~\ref{MaxPrinc}, the minimum principle holds for \erf{eqev}, but this contradicts 
that $\min_{\gO} Z^-=-1$.  Thus, we have $\sup_{n\in\N}\|(z_n)_*\|_\infty<+\infty$.

For the sequence $(z_n)$, which is uniformly bounded in $\ol\gO$, we consider 
the upper and lower relaxed limits $z^+$ and $z^-$ defined, respectively, by 
\[
z^+(x)=\limsup_{n\to\infty}{\kern-3pt}^* z_n(x)=\inf_{r>0}\sup\{z_n(y)\,:\,
|y-x|<r,\ n>r^{-1}\}, 
\] 
and
\[
z^-(x)=\liminf_{n\to\infty}{\kern-3pt}_* z_n(x)=\sup_{r>0}\inf\{z_n(y)\,:\,
|y-x|<r,\ n>r^{-1}\},  
\]
and observe that $-\sup_{n\in\N}\|(z_n)_*\|_{\infty}\leq z^-\leq  z^+\leq W$ in $\ol\gO$ 
and that $z^+$ and $z^-$ are a subsolution and a supersolution of \erf{boundedsub0}, respectively. 

Similarly to \erf{bc-0} for $Z_n$, since $(z_n)$ is uniformly bounded in $\ol\gO$, 
we deduce that there is a constant $M_2>0$ such that $(z_n)_*(x)\geq -M_2d(x)$ for all 
$x\in\ol\gO$ and $n\in\N$, which implies that $z^-=z^+=0$ on $\pl\gO$. 
The proof is now complete. 
\end{proof}

We remark that defining $W^\ep$ from $\psi^\ep$ in the proof above  
is  a sort of supconvolution (see \cite{KI} for the use of this supconvolution 
in a different situation).

\begin{proof}[Proof of Proposition~\ref{supersol}]
The finiteness of $\overline{\mu}_k^-$ is a consequence of Corollary~\ref{corbound} which gives a precise estimate. 

For $n\in\N$ let us consider the equation
\begin{equation}\label{eq1supersol}
\Pmk(D^2w)+H(x,\grad w)+\left(\overline\mu_k^--\frac1n\right)w=1\quad\mbox{in}\ \Omega.
\end{equation}

For each $n\in\N$, by  Proposition~\ref{boundedsub}, 
there are a subsolution $v_n$ and a supersolution $w_n$ of 
\beq\label{eqevn+1}
\Pmk(D^2u)+H(x,\grad u)+\left(\overline\mu_k^--\frac1{n}\right)u= 1
\quad\mbox{in}\ \Omega,
\eeq
satisfying $w_n\leq v_n\leq  0$ in $\ol\Omega$ 
and $w_n=v_n=0$ on $\pl\gO$.

We claim that $\sup_{n\in\N}\|w_n\|_\infty=+\infty$. 
Suppose by contradiction that $\sup_{n\in\N}\|w_n\|_\infty<+\infty$. 
We choose $j\in\N$ large enough so that 
\[
\fr 1j \left(2\sup_{n\in\N}\|w_n\|_\infty+\ol\mu_k^-+\fr 1j\right)\leq 1, 
\]
which implies that, since $w_j\leq v_j\leq 0$, 
\[
\fr{2}{j} v_j-\fr 1j\left(\ol\mu_k^-+\fr 1j\right)\geq -1 \ \ \text{ in }\ol\gO,
\]
and, hence, $v_j-1/j$ is a subsolution of 
\[
\cP_k^-(D^2u)+H(x,\grad u)+(\ol\mu_k^-+\fr 1j)u=0 \ \ \text{ in }\gO.
\]
Since $v_j-1/j<0$ in $\ol\gO$, 
this contradicts the definition of $\ol\mu_k^-$ and proves that 
$\sup_{n\in\N}\|w_n\|_\infty=+\infty$.

Up to extracting a subsequence, we may 
assume that 
\[
\lim_{n\to+\infty}\|w_n\|_\infty=+\infty.
\]
We introduce bounded functions 
$z_n=
\frac{w_n}{\left\|w_n\right\|_\infty}$, solutions of
$$
\Pmk(D^2z_n)+H(x,\grad z_n)+\left(\overline\mu_k^--\frac1n\right)z_n\leq\frac{1}{\left\|w_n\right\|_\infty}\quad\mbox{in}\ \Omega.
$$
We set
$$
v(x):=\liminf_{n\to+\infty}{\kern-3pt}_*\,z_n(x)
\ \ \text{ for }x\in\ol\gO. 
$$ 
This is the lower half relaxed limit of $(z_n)$ and is a supersolution of 
$\Pmk(D^2v)+H(x,\grad v)+\mu_k^- v\leq 0$ in $\Omega$. Moreover, it is clear that 
$v\leq 0$ in $\ol\gO$ and $\min_{\ol\gO} v=-1$. Using again the bound \eqref{barrier1}, 
we deduce that $v=0$ on $\pl\gO$, and the proof is complete. 
\end{proof}

\begin{proof}[Proof of Theorem~\ref{thm3}]  We begin by proving the following

{\bf Claim.} \emph{ For $\mu<\mu_k^-$  the operator $\Pmk(D^2\cdot)+H(x,\grad \cdot)+\mu\cdot$ satisfies the minimum principle.}

The proof proceeds like the proof of Theorem~\ref{MaxPrinc}, the only difference is that for $\rho\in(\mu,\mu_k^-)$, the $\limsup_{x\to z}u(x)$ could be zero for some $z\in\partial\Omega$. But using (\ref{SC2}) and the negativity of $u(x)$ we get
$$\Pmk(D^2u)+H(x,\grad u)\geq0\quad\text{in $\Omega$}$$
while
$$
\Ppk(D^2(-v))-H(x,-\grad(-v))\geq-\mu\left\|v^-\right\|_\infty\quad\text{in $\Omega$}.
$$
In view of Propositions \ref{Hopf}-\ref{upbound}, with $m=\mu\left\|v^-\right\|_\infty$, there exist two positive constants $C_1$ and $C_2$ such that
$$u(x)\leq-C_1d(x)\;\;\;\text{and}\;\;-v(x)\leq C_2d(x)\;\;\text{for any $x\in\Omega$}.$$
Hence 
$$
0<\frac{v(x_0)}{u(x_0)}\leq\gamma:=\sup_{x\in\Omega}\frac{v(x)}{u(x)}\leq\frac{C_2}{C_1}<+\infty.
$$
Now we can proceed exactly as in the proof of Theorem~\ref{MaxPrinc} in order to complete the proof of
 the claim.

To finish the proof  of Theorem~\ref{thm3} we observe that  Proposition~\ref{supersol} and the claim imply that $\overline{\mu}_k^-
\geq \mu_k^-$, but the reverse inequality is true by definition. 
\end{proof}

\begin{remark}{\normalfont
The bound (\ref{bounds1-eigenvalue}) clearly holds for $\mu_k^-$ under the assumptions of Theorem~\ref{thm3}.\\
Since $\mu^-_k\geq\overline\mu_k^-$, by definition, the inequality (\ref{bounds2-eigenvalue})  is a fortiori true for $\mu_k^-$.
Moreover (\ref{bounds2-eigenvalue}) is trivial for $bR_2\geq k$. We show in the Example  \ref{mu=0} that $\overline\mu_k^-$ can be zero.
}
\end{remark}

\begin{remark}\label{equality}
{\normalfont
The equality $\mu^-_k=\overline\mu_k^-$ holds true also in some non-convex case, for instance if $\Omega$ is a star-shaped domain, i.e. 
\begin{equation}\label{star-shaped}
\overline{\Omega-\left\{x_0\right\}}\subseteq(1+\varepsilon)(\Omega-\left\{x_0\right\})
\end{equation}
for some $x_0\in\Omega$ and all $\vep>0$. That was noticed e.g. in \cite{P} in the case of the  Pucci's extremal uniformly elliptic operators. Supposing $x_0=0$, for any $\varepsilon>0$ there exists, by definition, $w_\varepsilon<0$ in $\Omega$ satisfying
$$\Pmk(D^2w_\varepsilon)+H(\grad w_\varepsilon)+(\mu^-_k-\varepsilon)w_\varepsilon\geq0.$$
Hence $v_\varepsilon(x)=w_\varepsilon\left(\frac{x}{1+\varepsilon}\right)$ is negative in $\overline\Omega$ and if
\begin{equation}\label{Hrestriction}
H=H(\xi)=H^+(\xi),
\end{equation}
then
$$\Pmk(D^2v_\varepsilon)+H(\grad v_\varepsilon)+\frac{\mu^-_k-\varepsilon}{(1+\varepsilon)^2}v_\varepsilon\geq0\quad\;\text{in  $\Omega$}.$$ 
In this way $$\frac{\mu^-_k-\varepsilon}{(1+\varepsilon)^2}\leq\overline\mu_k^-$$
and $\mu^-_k=\overline\mu_k^-$ in the limit $\varepsilon\to0$. The same holds true for $\mu^+_k$ and $\overline\mu_k^+$ when $H=H(\xi)=-H^-(\xi)$.\\
Note that on one hand the class of the bounded domains satisfying \eqref{star-shaped} strictly includes $\mathcal C$, but on the other hand the equality $\mu^-_k=\overline\mu_k^-$ is here realised under the restriction \eqref{Hrestriction}, while in Theorem \ref{thm3} the Hamiltonian $H$ is allowed to be negative and dependent on the $x$-variable.}
\end{remark}

\subsection{Some unusual phenomena.}\label{unusual}
It is well known (see e.g. \cite{BNV}) that in the uniformly elliptic case the principal eigenvalues tend to infinity when the measure of the domain tends to zero; the next example shows that this is not necessarily the case for $\Pmk$.

\begin{example}\label{mu=0} \rm We show  that in an annulus $\overline{\mu}_k^-=0$, even if 
the measure of the annulus tends to zero, as long as the diameter is sufficiently large.
For $k<N$, the radial function
$$v(x)=\sin|x|+\cos\varepsilon$$
is a supersolution of the problem
\begin{equation*}
\left\{\begin{array}{lc}
\Pmk(D^2v)-b|\grad v|=0 &\mbox{in}\ A_\varepsilon=B_{\frac32\pi+\varepsilon}\backslash\overline{B}_{\frac32\pi-\varepsilon}\\
v=0 &\mbox{on}\ \partial A_\varepsilon
\end{array}
\right.
\end{equation*}
where $b=\frac{k}{\frac32\pi}$ and $\varepsilon$ is small enough (see \cite{GV}). Since $v$ violates the minimum principle, being negative in the annulus $A_\varepsilon$, we deduce from i) of Theorem \ref{MaxPrinc} that
$$\overline \mu^-_k=0.$$
\end{example}

\medskip

In the next example we show how the definition of $\overline\mu^+_k$ is strongly unstable with respect to 
perturbations both of the operator and the domain.
\begin{example}
\rm
Let $\Omega=B_R$. For $k<N$ and $n\in\N$, the values  
$\overline \mu^+_k$ associated to the operators
$$\Pmk(\cdot)+\frac{k}{R+\frac1n}|\cdot|$$
blows-up to $+\infty$ in view of Proposition~\ref{infinito}, since in this case $\frac{k}{R+\frac1n} R<k$. Moreover 
$$\Pmk(\cdot)+\frac{k}{R+\frac1n}|\cdot|\longrightarrow\Pmk(\cdot)+\frac kR|\cdot|$$
as $n\to+\infty$ locally uniformly in $\Sn\times\Rn$. On the other hand, taking the function $w(x)=\left(R^2-|x|^2\right)^\gamma$ with  $\gamma>1$, it turns out that
\begin{equation*}
\begin{split}
\Pmk(D^2w)+\frac kR|\grad w|+\frac{2\gamma k}{R^2}w&=\left(R^2-|x|^2\right)^{\gamma-1}\left(-2\gamma k+2\frac kR\gamma|x|+\frac{2\gamma k}{R^2}(R^2-|x|^2)\right) \geq 0;
\end{split}
\end{equation*}
moreover $w=0$ on $\partial\Omega$, $w>0$ in $\Omega$ and so $\overline \mu^+_k\leq \frac{2\gamma k}{R^2}$ by ii) of Theorem \ref{MaxPrinc}.

Concerning the instability with respect to small perturbations of $\Omega$, we consider the sequence of expanding subdomains $\Omega_n=B_{R-\frac1n}$ and the operator $\Pmk(\cdot)+\frac kR|\cdot|$. As before, for any $\Omega_n$ one has $\overline \mu^+_k=+\infty$, while in $\overline \mu^+_k\leq1$ in the limit case $\Omega=\cup_{n\in\N}\Omega_n$.  

\end{example} 
 Notice that in \cite{BNV} the stability of the principal eigenvalue with respect to interior perturbations of the domain is proved by means of the Krylov-Safonov Harnack inequality. It is not surprising therefore to expect the failure of the Harnack inequality in our degenerate setting, which is indeed the case as can be seen in the following very simple example. The nonnegative function $u(x_1,\ldots,x_N)=x_N^2$ is clearly a solution of $\Pmk(D^2u)=0$ in $B_1$ for $k<N$, but $\displaystyle\sup_{B_1} u=1$ and $\displaystyle\inf_{B_1}u=0$.

Other examples of instability are provided in \cite{BCPR} for first order operators.

\section{Existence } \label{existence}

In this section we shall prove existence results for Dirichlet problems
\begin{equation}\label{dirk}
\left\{\begin{array}{lc}
\Pmk(D^2u)+H(x,\grad u)+\mu u=f(x) &\mbox{in}\ \Omega\\
u=0 &\mbox{on}\ \partial\Omega,
\end{array}
\right.
\end{equation}
with $\Omega$  in the class $\cC_R$.
We start with the case where $k$ is any number between 1 and $N$.

\begin{proposition}Assume \erf{SC2} -\erf{SC3}. Let $\Omega\in\cC_R$ and $\mu<\ol\mu_k^-$, 
and assume that $bR<k$. 
If $f$ is bounded and $H$ satisfies, for all $x\in\Omega$ and for all $\xi,\ \eta$ in $\R^N$,
\beq\label{HH} |H(x,\xi)-H(x,\eta)|\leq b|\xi-\eta|,\ \eeq 
then for all
$$\mu<{\mu}_{k,b}^-:=\sup\{\mu\in\R: \exists w<0 \   \mbox{ in }\Omega,\  \Pmk(D^2w)-b|\grad w|+ \mu w\geq 0\ \mbox{in}\ \Omega\},$$
there exists a unique solution of \eqref{dirk}.
\end{proposition}
\begin{proof} Let $v$ and $w$ be as in Proposition~\ref{boundedsub}. By \eqref{HH}, the nonpositive function $u=w-v$ is a supersolution of
$\Pmk(D^2u)-b|\grad u|+ \mu u= 0$ (see \cite{GV}). Using Theorem~\ref{thm3} i), we get that $u\geq 0$. Hence $v=w$ is the required solution.
\end{proof}

In the rest of the section we shall only consider the case $k=1$, in that case beside the existence below the generalized eigenvalue we can also prove existence of the eigenfunction. The proofs somehow follow the schemes of \cite{BD, BD2}.

\begin{theorem}\label{exi}
Let $\gO\in\cC_R$,  let $H$ satisfying \eqref{SC2}-\eqref{SC3} and let $f$ be a bounded continuous function in $\Omega$.  Assume $bR<1$. 
Then   
there exists a solution $u\in{\rm Lip}(\overline\Omega)$ of
\begin{equation}\label{dir1}
\left\{\begin{array}{lc}
\Pmo(D^2u)+H(x,\grad u)+\mu u=f(x) &\mbox{in}\ \Omega\\
u=0 &\mbox{on}\ \partial\Omega,
\end{array}
\right.
\end{equation}
in the following two cases:
\begin{itemize}
	\item[i)] for  $\mu< \mu_1^-$;
	\item[ii)]  for any $\mu$ if $f\leq0$.
\end{itemize}
\end{theorem}
\def\bcases{\begin{cases}}
\def\ecases{\end{cases}}
The proof uses the construction in Proposition~\ref{boundedsub} and the global Lipschitz regularity obtained in Proposition~\ref{Lip-regularity} for subsolutions.
\begin{proof}[Proof of Theorem \ref{exi}] $\,$ 
We first consider the case where $\mu<\mu_1^-$. By Theorem~\ref{thm3} and Proposition~\ref{boundedsub}, we see that there are a subsolution $v$ and a supersolution $w$ of \eqref{dir1}
such that $w\leq v$ in $\ol\gO$. 
By estimate \erf{barrier1}, there is a constant $C>0$ such that $-Cd\leq w\leq v$ 
in $\ol\gO$. 

As in Proposition~\ref{boundedsub}, the standard construction of barrier functions for elliptic PDE yields 
a supersolution $W\in C(\ol\gO)$ of \erf{dir1} that satisfies $W=0$ on $\pl\gO$ 
and $W\geq 0$ in $\ol\gO$. If $f\geq 0$  then just take $W\equiv 0$.

We define function $u$ in $\ol\gO$ through the Perron procedure, that is, 
\[
u(x)=\sup\{z(x)\,:\, z\text{ subsolution of \erf{dir1}},\ v\leq u\leq W \ \text{ in }\ol\gO\}.
\]
The upper semicontinuous envelope $u^*$ is a subsolution of \erf{dir1} and satisfies 
$v\leq u^*\leq W$ in $\ol\gO$, which implies that $u=u^*$ in $\ol\gO$ and, hence, 
$u$ is upper semicontinuous in $\ol\gO$. Since $u\geq- Cd$ and $u=0$ on $\pl\gO$, 
by Proposition~\ref{Lip-regularity}, we see that $u$ is Lipschitz continuous in $\ol\gO$. 
Hence $u=u_*$  and it is a supersolution of \erf{dir1}, we conclude the proof of i).

For the proof of ii), we can treat the case where $f\leq 0$ in $\ol\gO$ in the same way. The only difference is 
that, when $f\leq 0$, the constant function $0$ is a subsolution of \erf{dir1} 
and replaces $v$ in the argument above.  Thus, the bound on $\mu$ is not needed and the resulting solution $u$ is nonnegative. 
\end{proof}

\begin{theorem}\label{eigenfunction}
 Let $\Omega$, $H$ and $b$ as in the Theorem~\ref{exi}. Then there exists a negative function $\psi_1\in{\rm Lip}(\overline\Omega)$ such that
\begin{equation}\label{1dir}
\left\{\begin{array}{lc}
\Pmo(D^2\psi_1)+H(x,\grad\psi_1)+\mu_1^-\psi_1=0 &\mbox{in}\ \Omega\\
\psi_1=0 &\mbox{on}\ \partial\Omega.
\end{array}
\right.
\end{equation}
\end{theorem}

\begin{proof}
Let $\mu_n\nearrow\mu^-_1$ and use Theorem~\ref{exi} to build $u_n\in {\rm Lip}(\overline\Omega)$ a solution of 
\begin{equation}\label{dir5}
\left\{\begin{array}{lc}
\Pmo(D^2u_{n})+H(x,\grad u_{n})+\mu_n u_n=1 &\mbox{in}\ \Omega\\
u_{n}=0 &\mbox{on}\ \partial\Omega.
\end{array}
\right.
\end{equation}
Observe that $u_n$ are nonnegative because the forcing term being positive in Perron's construction  we can use zero as the supersolution that bounds from above.
 
We claim that $\lim_{n\to\infty}\left\|u_n\right\|_\infty=+\infty$. Assume by contradiction that $\sup_{n\in\N}\left\|u_n\right\|_\infty<+\infty$. By Proposition~\ref{Lipschitz_regularity} the sequence $(u_n)_{n\in \N}$ is bounded in ${\rm Lip}(\overline\Omega)$ and converges , up to some subsequence, to a nonpositive solution $u$ of
\begin{equation*}
\left\{\begin{array}{lc}
\Pmo(D^2u)+H(x,\grad u)+\mu_1^- u=1 &\mbox{in}\ \Omega\\
u=0 &\mbox{on}\ \partial\Omega.
\end{array}
\right.
\end{equation*}
The function $u$ is negative in $\Omega$, otherwise if $\max_{x\in\overline\Omega}u=u(x_0)=0$ and $x_0\in\Omega$, then $\varphi(x)=0$ should be a test function touching $u$ from above in $x_0$ and therefore $0\geq1$.\\
Hence, for small positive $\varepsilon$, we have
$$\Pmo(D^2u)+H(x,\grad u)+(\mu_1^-+\varepsilon) u\geq0\quad\;\text{in $\Omega$}$$
contradicting the maximality of $\mu_1^-$.

\medskip
For $n\in\N$ the functions $v_n=\frac{u_n}{\left\|u_n\right\|_\infty}$ satisfy
\begin{equation}\label{dir6}
\left\{\begin{array}{lc}
\Pmo(D^2v_{n})+H(x,\grad v_{n})+\mu_n v_n=\frac{1}{\left\|u_n\right\|_\infty} &\mbox{in}\ \Omega\\
v_{n}=0 &\mbox{on}\ \partial\Omega
\end{array}
\right.
\end{equation}
and are bounded in ${\rm Lip}(\overline\Omega)$, again by means of Proposition~\ref{Lipschitz_regularity}. Extracting a subsequence if necessary, $(v_n)_{n\in\N}$ converges uniformly to a nonpositive function $\psi_1$ such that $\left\|\psi_1\right\|_\infty=1$. Taking the limit as $n\to+\infty$ in (\ref{dir6}) we have  
\begin{equation*}
\left\{\begin{array}{lc}
\Pmo(D^2\psi_1)+H(x,\grad\psi_1)+\mu_1^-\psi_1=0 &\mbox{in}\ \Omega\\
\psi_1=0 &\mbox{on}\ \partial\Omega.
\end{array}
\right.
\end{equation*}
By the strong maximum principle (see Remark~\ref{strong_MP}), we conclude $\psi_1<0$ in $\Omega$ as we wanted to show.
\end{proof}

\medskip

We conclude by computing explicitly the principal eigenvalue and eigenfunction for $\Pmo$, with $H=0$, in the ball $B_R$. We first note that $\overline\mu^-_1=\mu_1^-$, as a consequence of Theorem \ref{thm3} or, equivalently, of Remark \ref{equality}.\\
The function
$$\psi_1(x)=-\cos\left(\frac{\pi}{2R}|x|\right)$$
is twice differentiable everywhere, negative in $B_R$ and zero on $\partial B_R$. The ordered eigenvalues of the Hessian matrix are
\begin{equation*}
\begin{split}
\lambda_1\left(D^2\psi_1(x)\right)&=\left(\frac{\pi}{2R}\right)^2\cos\left(\frac{\pi}{2R}|x|\right)\\ 
\lambda_2\left(D^2\psi_1(x)\right)&=\ldots=\lambda_N\left(D^2\psi_1(x)\right)=\left(\frac{\pi}{2R}\right)\frac{\sin\left(\frac{\pi}{2R}|x|\right)}{|x|},
\end{split}
\end{equation*}
if $x\neq0$ and
$$
\lambda_1\left(D^2\psi_1(0)\right)=\ldots=\lambda_N\left(D^2\psi_1(0)\right)=\left(\frac{\pi}{2R}\right)^2,
$$
so that 
$$
\Pmo\left(D^2\psi_1(x)\right)+\left(\frac{\pi}{2R}\right)^2\psi_1(x)=0\quad\;\text{in $\Omega$}.
$$
In particular $\psi_1$ is a negative subsolution of $\Pmo(D^2\cdot)+\left(\frac{\pi}{2R}\right)^2\cdot=0$, hence by definition of $\mu_1^-$ we have $\mu_1^-\geq\left(\frac{\pi}{2R}\right)^2$. On the other hand the function $\psi_1$ invalidates the minimum principle and we get also the reversed inequality $\mu_1^-\leq\left(\frac{\pi}{2R}\right)^2$ by means of Theorem \ref{thm3}. In this way
$$
\mu_1^-=\left(\frac{\pi}{2R}\right)^2
$$
and $\psi_1$ is a negative radial eigenfunction.\\
It is worth to point out that for the 1-homogeneous infinity Laplacian $\Delta_\infty u=\left\langle D^2u\frac{\grad u}{|\grad u|},\frac{\grad u}{|\grad u|}\right\rangle$, one has
$$\overline\mu_1^+=\left(\frac{\pi}{2R}\right)^2$$
with $\varphi_1(x)=\cos\left(\frac{\pi}{2R}|x|\right)$ positive eigenfunction (see \cite[Section 4]{J}). In our framework we have on the contrary $\overline\mu_1^+=+\infty$ in view of Proposition~\ref{infinito}.

\section{Strictly convex domains, a characterization.}\label{convex} In this section we give the proof of Proposition~\ref{hulahoop} which we like to refer to as Proposition hula hoop.

\def\stm{\setminus}
\def\baligned{\begin{aligned}}
\def\ealigned{\end{aligned}}

 \def\gD{\varDelta}
 \def\gk{\kappa}
\def\bproof{\begin{proof}}
\def\eproof{\end{proof}}
 
We begin with a technical lemma.
\begin{lemma} \label{HH1}Let $\gO$ be a non-empty bounded and open 
subset of $\R^N$, with $C^2$-boundary, and $p\in\pl\gO$. Let $\nu(x)$ denote 
the outward normal unit vector of $\gO$ at $x\in\pl\gO$. Assume that $N>2$, and let 
$H\subset\R^N$ be a $2$-dimensional plane
passing through $p$ which is not perpendicular to $\nu(p)$. Set $\gD=\gO\cap H$. 
Let $H$ have the Euclidean structure induced by $\R^N$.

\noindent \emph{i)} Then, $\gD$ is a non-empty bounded and open subset, 
with $C^2$-boundary, of the plane $H$.\\
\noindent \emph{ii)} Assume in addition that the principal curvatures, $\gk_1,\ldots,\gk_{N-1}$,  of $\pl\gO$ at $p$   
are positive.  

Then, the curvature of the planar curve $\pl_H\gD$ at $p$ is 
bounded from below by $\min_{1\leq i< N}\gk_i$, where $\pl_H A$ denotes the boundary 
of $A\subset H$, relative to $H$.    
\end{lemma}

In the above, the perpendicularity of $H$ and $\nu(p)$ may be expressed as the condition 
that $\nu(p)\cdot(q-p)=0$ for all $q\in H$. 

\bproof We first prove i). We choose two orthonormal vectors $e_1,\,e_2\in\R^N$ so that 
$H=\{p+x_1e_1+x_2e_2\,:\, x_1,\, x_2\in\R\}$. By the non-perpendicularity of $H$ and $\nu(p)$, 
we may assume that $\nu(p)\cdot e_1<0$. 

Since $\gO$ has $C^2$-boundary, if $\gd>0$ is small enough, then 
$p+\gd e_1\in\gO$ and  $p+\gd e_1$ is  an interior point of $\gD$, relative to $H$. 
Since $\gO$ is open, $\gD$ is open relative to $H$. Hence, $\gD$ is a non-empty 
open subset of $H$. It is clear that $\gD$ is convex since it is an intersection of 
two convex sets and also that $\gD$ is bounded. 

Now, we show that $\gD$ is a domain, with $C^2$-boundary, in $H$. 
It is obvious that $\pl_H\gD\subset H\cap\pl\gO$. Fix any $q\in H\cap\pl\gO$. 
We consider the function $\rho\in C(\R^N)$ given by 
\[
\rho(x)=
\bcases \dist(x,\pl\gO)&\text{ if }x\in\gO,\\[3pt] 
-\dist(x,\pl\gO)&\text{ if }x\in\R^N\setminus \gO.
\ecases
\]
This function $\rho$ is $C^2$ near the boundary $\pl\gO$ and $\grad \rho(x)=-\nu(x)$ 
for all $x\in\pl\gO$. 
Set $p_\gd=p+\gd e_1\in\gD$, note that $\rho(p_\gd)>0$, and choose $(a,b)\in\R^2$ so that $q=p_\gd+ae_1+be_2$.
By the concavity of $\rho$, we find that for any $t\in[0,\,1]$,
\[
\rho(p_\gd+t(ae_1+be_2))=\rho((1-t)p_\gd+tq)\geq (1-t)\rho(p_\gd)+t\rho(q)=(1-t)\rho(p_\gd),
\]
and, hence,
\[
\fr{d}{dt}\rho(p_\gd+t(ae_1+be_2))\Big|_{t=1}\leq-\rho(p_\gd)<0,
\]
which shows that
\[
0>\grad \rho(q)\cdot (ae_1+be_2)=-\nu(q)\cdot(ae_1+be_2).
\]
Noting that 
\[H\cap\pl\gO=\{p_\gd+x_1e_1+x_2e_2\,:\, (x_1,x_2)\in\R^2, \ \rho(p_\gd+x_1e_1+x_2e_2)=0\}
\] 
and applying the implicit function theorem to the function: $\R^2\ni(x_1,x_2)\mapsto 
\rho(p_\gd+x_1e_1+x_2e_2)$, we see that, in a neighborhood of $q$, 
$H\cap\pl\gO$ is a $C^2$-curve in $H$ and that $q\in\pl_H\gD$.  
Because of the arbitrariness of $q\in H\cap \pl\gO$, 
we find that $H\cap\pl\gO$ is a $C^2$-curve in $H$ and also that 
$H\cap\pl\gO\subset\pl_H\gD$. Thus, we conclude that 
$\pl_H\gD=H\cap\pl\gO$ and that $\gD$ has $C^2$-boundary in $H$.  

Next, we prove (ii). We may assume by translation and orthogonal transformation that $p=0$
and $\nu(p)=(0,\ldots,0,-1)$. We can choose a neighborhood $V\subset\R^N$ of $p=0$,
a neighborhood $U\subset \R^{N-1}$ of $0\in\R^{N-1}$ and a function $g\in C^2(U)$ 
such that for any $x=(x_1,\ldots,x_N)\in V$,
\[
x\in \gO \ \ 
\text{ if and only if } \ \ 
(x_1,\ldots,x_{N-1})\in U \ \text{ and }\ x_N>g(x_1,\ldots,x_{N-1}). 
\]
We have $g(0)=0$, $\grad g(0)=0$ and we 
may assume further that $D^2g(0)=\mathrm{diag}(\gk_1,\ldots,\gk_{N-1})$.
We choose $R>0$ so that $1/R<\min_{1\leq i<N}\gk_{i}$, and consider the open ball 
$B$ with center at $-R\nu(p)=(0,\ldots,0,R)$ and radius $R$. We may assume 
by replacing $U$ and $V$ by smaller ones (in the sense of inclusion), if necessary,
that for any $x\in V$,
\[
x\in B \ \ \text{ if and only if } \ \ 
(x_1,\ldots,x_{N-1})\in U \ \text{ and }\ x_N>f(x_1,\ldots,x_{N-1}),
\] 
where $f(x_1,\ldots,x_{N-1})=R-\sqrt{R^2-(x_1^2+\cdots+x_{N-1}^2)}$. 
Note that $\grad f(0)=0$ and $D^2f(0)=(1/R)I$, where $I$ denotes the identity matrix of order 
$n-1$.  By Taylor's theorem, we may assume again by replacing $U$ and $V$ by smaller 
ones, if necessary, that $f(y)<g(y)$ for all $y\in U\setminus\{0\}$. This yields
\[
V\cap \gO\subset V\cap B,
\]
which shows that
\[
V\cap \gD\subset V\cap B\cap H.
\] 
Thus, observing that $\pl B\cap H=\pl_H(B\cap H)$, which is a special case of the identity, 
$\pl\gO\cap H=\pl_H\gD$, with $B$ in place of $\gO$,
that $B\cap H$ is a non-empty, planar, open disk with radius smaller than or equal to $R$ 
and that $p=0\in \pl_H\gD\cap \pl_H (B\cap H)$, we  conclude that
the curvature of the planar curve $\pl_H\gD$ at $p$ is larger than or equal to $1/R$. 
This completes the proof.   
\eproof

\begin{lemma} \label{HH2} 
Let $\gO$ be a non-empty bounded and open 
subset, with $C^2$-boundary, of $\R^N$. Let $\gk>0$ is a lower bound of the principal curvatures of $\pl\gO$ at every point $x\in\pl\gO$. Set $R=1/\gk$. Then, for any $z\in\pl \gO$, we have
\beq\label{HH2-0}
\gO\subset B_{R}(z-R\nu(z)). 
\eeq  
\end{lemma} 

Clearly, \erf{HH2-0} shows that $\gO\in\cC_R$. Indeed
we have proved that
\[
\Omega\subset \bigcap_{z\in \partial\Omega}B_R(z-R\nu(z)).
\]

On the other hand, by the convexity of $\Omega$, we have 
\[
\Omega=\bigcap_{z\in\partial\Omega}\{x\in\Rn\,:\, (x-z)\cdot\nu(z)< 0\}.
\]
Observe that for any $z\in\partial\Omega$,
\[
B_R(z-R\nu(z))\subset \{x\in\Rn\,:\, (x-z)\cdot\nu(z)<0\}.
\]
Indeed, if $x\in B_R(z-R\nu(z))$, then 
\[
R^2>|x-z+R\nu(z)|^2=|x-z|^2+2R(x-z)\cdot\nu(z)+R^2> 2R(x-z)\cdot\nu(z)+R^2,
\]
and 
\[
(x-z)\cdot\nu(z)< 0.
\]
Thus, 
\[
\Omega\supset\bigcap_{z\in\partial\Omega}B_R(z-R\nu(z)).
\]
In conclusion the Lemma~\ref{HH2} above proves 
Proposition~\ref{hulahoop}.

\bproof It is enough to show that for any $M>R$ and $z\in\pl\gO$, 
\beq\label{1}
\gO\subset B_M(z-M\nu(z)). 
\eeq

We fix any $M>R$ and $p\in\pl\gO$. To show \eqref{1}, we suppose to the contrary that 
\erf{1} does not hold, and will get a contradiction. 

We can thus choose a point  
$q\in\gO\setminus B_{M}(p-M\nu(p))$. 

Select $m>0$ so small that $r:=p-m\nu(p)\in\gO\cap B_M(p-M\nu(p))$. 
Note that the line segment 
$[r,\,q]:=\{(1-t)r+tq\,:\, 0\leq t\leq 1\}$ is contained in the set 
$\gO$ and that $r\in B_M(p-M\nu(p))$  and $q\not\in B_M(p-M\nu(p))$. 
These 
imply that, for some $\tau\in(0,\,1]$, 
\[(1-\tau)r+\tau q\in \gO\cap \pl B_M(p-M\nu(p)).\] 
Replacing $q$ by $(1-\tau)r+\tau q$ if $\tau<1$, we may assume that 
$q\in\pl B_M(p-M\nu(p))$.

Since $\gO$ is open, we may assume by replacing $q$ by a nearby point, if needed, that
two vectors $\nu(p)$ and $q-p$ are linearly independent. In particular, we have 
$q\not=p$ and $q\not=p-2M\nu(p)$.   
Let $H$ be the plane passing through three points  
$p,\,q,\,p-M\nu(p)$. We set $\gD=\gO\cap H$ and $B_H=B_M(p-M\nu(p))\cap H$. 
Since $p-M\nu(p)\in H$, it is clear that $B_H$ is the planar open disk with center 
$p-M\nu(p)$ and radius $M$.

Fix $Q\in(R,\,M)$, so that $\gk>1/Q$.   
As in the proof of Lemma \ref{HH1} (ii), we can choose a neighborhood $V$ of $p$ 
so that 
\[
\gO\cap V \subset B_Q(p-Q\nu(p))\cap V, 
\] 
from which we find that
\beq\label{2}
\gD\cap V\subset H\cap B_Q(p-Q\nu(p)) \cap V. 
\eeq

We put $e_1=-\nu(p)$ and select 
a unit vector $e_2\in\R^N$, orthogonal to $e_1$, so that two vectors   
$e_1,\,e_2$ parallel to the plane $H$, that is, 
$H=\{p+x_1e_1+x_2e_2\,:\, x_1,x_2\in\R\}$. 

We select $(a,b)\in\R^2$ so that $q=p+ae_1+be_2$. Since  
$q\in\pl B_M(p-M\nu(p))\setminus\{p-2M\nu(p),\,p\}$, it follows that 
$0<a<2M$ and $b\not=0$. We may assume by replacing $e_2$ by $-e_2$, if needed, 
that $b<0$.

We set 
\[
\gD_2 =\{(x_1,x_2)\in\R^2\,:\, p+x_1e_1+x_2e_2\in\gD \},
\]
\[
g(x_1)=\inf\{x_2\in\R\,:\, (x_1,x_2)\in\gD_2 \}\ \ \text{ for }\ x_1\in(0,\,a]. 
\]
It is easily seen that $\gD_2$ is a strictly convex, bounded and open set, with $C^2$-boundary, in $\R^2$, 
that the line segment $\{t(a,b)\,:\, (0,\,1]\}$, 
connecting the origin and the point $(a,b)$, lies in the set $\gD_2$,  
that $g$ is locally Lipschitz continuous, convex function on $(0,\,a]$, 
and that the graph $\{(x_1,g(x_1))\,:\, x_1\in(0,a]\}$ 
is a subset of $\pl \gD_2$. The last two remarks together with the smoothness of $\gD_H$ 
implies that $g\in C^2((0,\,a])$.  

We consider the function $f_M\in C([0,\,a])$ defined by 
\[
f_M(x_1)=-\sqrt{M^2-(x_1-M)^2}.
\]
Obviously we have, for $(x_1,x_2)\in(0,\,a]\tim\R$,
\[
x_2>f_M(x_1) \ \ \ \text{ if }\ p+x_1e_1+x_2e_2\in B_H. 
\]
Similarly, we define $f_Q\in C([0,\,2Q])$ by
\[
f_Q(x_1)=-\sqrt{Q^2-(x_1-Q)^2}.
\]
By \erf{2}, if we define the function $h$ on $[0,\,a]$ by 
\[
h(x)=\bcases
0& \text{ if }x=0,\\[3pt]
g(x)&\text{ if }x\in(0,\,a],
\ecases
\]
then $f_M(0)=f_Q(0)=h(0)=0$ and $f_M(x_1)<f_Q(x_1)\leq h(x_1)$ for all $x_1\in(0,\,\gd]$ 
and some small $\gd>0$.   On the other hand, since 
$\{x_1(a,b)\,:\, x_1\in(0,\,1]\}\subset \gD_2$, we have $h(x_1)=g(x_1)\leq (b/a)x_1$ 
for all $x_1\in(0,\,a]$. It is now clear that $h\in C([0,\,1])$.  

Since $q=p+ae_1+be_2\in\gD\cap\pl B_M(p-M\nu(p))$, we have $h(a)=g(a)<b=f_M(a)$. 
Consider the function $\phi\in C([0,\,a])$ given by 
\[
\phi(x)=h(x)-f_M(x).   
\]
It follows that $\phi(0)=0$, $\phi(a)<0$ and $\phi(\gd)>0$. 
Accordingly, $\phi$ has a positive maximum at a point $d\in(0,\,a)$.
Hence, $\phi'(d)=0$ and $\phi''(d)\leq 0$. That is, we have 
$f_M'(d)=g'(d)$ and $f_M''(d)\geq g''(d)$, which shows that the curvature of 
the graph $g$ at $(d, g(d))$ is smaller than or equal to that of $f_M$, which is $1/M$. 
This shows that the planar curve $\pl_H\gD$ has curvature smaller than $1/R$ at 
$p+de_1+g(d)e_2$. Since the planar curve $\pl_H\gD$ has curvature larger than or equal to 
$\gk=1/R$ by Lemma \ref{HH1}, this is a contradiction.   
\eproof


\begin{thebibliography}{99}
\bibitem{AS} L. Ambrosio, H. M. Soner, \emph{Level set approach to mean curvature flow in arbitrary codimension},  J. Differential Geom. 43 (1996), 693-737. 
\bibitem{AGV} M.E. Amendola, G. Galise, A. Vitolo, \emph{Riesz capacity, maximum principle, and removable sets of fully nonlinear second-order elliptic operators},  Differential Integral Equations 26 (2013), 845-866.
\bibitem{Ar}
S. N. Armstrong, 
\emph{Principal eigenvalues and an anti-maximum principle for
homogeneous fully nonlinear elliptic equations},
J. Differential Equations,
246 (2009), 2958--2987.
\bibitem {BCPR} H. Berestycki, I. Capuzzo Dolcetta, A. Porretta, L. Rossi, \emph{Maximum Principle and generalized 
principal eigenvalue for degenerate elliptic operators},  J. Math. Pures Appl. (9) 103 (2015), no. 5, 1276-1293.
\bibitem{BNV} H. Berestycki, L. Nirenberg, S. Varadhan, \emph{The principle eigenvalue and maximum principle for second
order elliptic operators in general domains}, Comm. Pure. Appl. Math 47 (1) (1994) 47-92.
\bibitem{BD} I. Birindelli, F. Demengel, \emph{Eigenvalue, maximum principle and regularity for fully non linear 
homogeneous operators}, Comm. Pure Appl. Anal. 6 (2) (2007) 335-366.
\bibitem{BD2} I. Birindelli, F. Demengel, \emph{The Dirichlet problem for singular fully nonlinear operators}, Discrete Contin. Dyn. Syst. 2007, Dynamical systems and differential equations. Proceedings of the 6th AIMS International Conference, suppl., 110-121. 
\bibitem{Bo} J.M. Bony, \emph{ Principe du maximum, in\'egalite de Harnack et unicit\'e du probl\`eme de Cauchy pour les op\'erateurs elliptiques d\'eg\'en\'er\'es. } Ann. Inst. Fourier  19 (1) (1969), 277-304.
\bibitem{BEQ} J. Busca, M.J. Esteban, A. Quaas, {\em Nonlinear eigenvalues and bifurcation
problems for Pucci's operator}, Ann. Inst. H. Poincar\'e Anal. Non Lin\'eaire  {\bf 22} (2005), 187--206.
\bibitem{CLN1} L. Caffarelli, Y. Y. Li, L. Nirenberg, \emph{Some remarks on singular solutions of nonlinear elliptic equations. I}, J. Fixed Point Theory Appl. 5 (2009), 353-395.
\bibitem{CLN3} L. Caffarelli, Y. Y. Li, L. Nirenberg, \emph{Some remarks on singular solutions of nonlinear elliptic equations III: viscosity solutions including parabolic operators. },  Comm. Pure Appl. Math. 66 (2013), 109-143.
\bibitem{CDLV} I. Capuzzo Dolcetta, F. Leoni, A. Vitolo, \emph{ On the inequality $F(x,D^2u)\geq f(u)+g(u)|Du|^q$.} Math. Ann. 365 (2016), no. 1-2, 423-448.
\bibitem{Cprimer} M.G. Crandall, \emph{Viscosity solutions: a primer}, Viscosity solutions and applications (Montecatini Terme, 1995), 1-43, Lecture Notes in Math., 1660, Springer, Berlin, 1997.
\bibitem{CIL} M.G. Crandall, H. Ishii, P.-L. Lions, \emph{User's guide to viscosity solutions of second order partial differential
equations}, Bull. Amer. Math. Soc. (N.S.) 27 (1) (1992) 1-67.
\bibitem{CutLe}A. Cutr\`\i, F. Leoni, \emph{On the Liouville property for fully nonlinear equations,} Ann. Inst. H. Poincar\'e Anal. Non Lin\'eaire 17, 2 (2000) 219-245.
\bibitem{GV} G. Galise, A. Vitolo, \emph{Removable singularities for degenerate elliptic Pucci operators}, accepted in Adv. Differential Equations.
\bibitem{GT} D. Gilbarg and N.S. Trudinger, {\it Elliptic Partial Differential Equations of Second Order}, $2^{\rm nd}$ 
ed., Grundlehren Math. Wiss. 224, Springer-Verlag, Berlin-New York (1983).
\bibitem{HL1}  F.R. Harvey, H.B. Jr. Lawson,\emph{Dirichlet duality and the nonlinear Dirichlet problem}, Comm. Pure Appl. Math. 62 (2009),  396-443.
\bibitem{HL2}  F.R. Harvey, H.B. Jr. Lawson, \emph{$p$-convexity, $p$-plurisubharmonicity and the Levi problem},  Indiana Univ. Math. J. 62 (2013), 149-169.
\bibitem{IY} H. Ishii, Y. Yoshimura, {\emph Demi-eigenvalues for uniformly elliptic Isaacs
operators}, preprint.
\bibitem{J} P. Juutinen, \emph{Principal eigenvalue of a very badly degenerate operator and applications}, J. Differential 236 (2) (2007) 532-550.
\bibitem{KI} I. C. Kim, \emph{A free boundary problem arising in flame propagation}, 
J. Differential Equations 191 (2003), 470--489.
\bibitem{KN} J.J. Kohn,  L. Nirenberg, \emph{
Degenerate elliptic-parabolic equations of second order},
Comm. Pure Appl. Math. 20 (1967) 797-872. 
\bibitem{L} P.-L. Lions, {\emph Bifurcation and optimal stochastic control}, Nonlinear Anal. 7 (1983), 177-207.
\bibitem{OS} A. M. Oberman, L. Silvestre, \emph{The Dirichlet problem for the convex envelope}, Trans. Amer. Math. Soc. 363 (2011), 5871-5886.
\bibitem{P} C. Pucci, \emph{Maximum and minimum first eigenvalues for a class of elliptic operators}, 
Proc. Amer. Math. Soc. 17 (1966) 788-795.
\bibitem{QS}  A. Quaas, B. Sirakov,  \emph{Principal eigenvalues and the Dirichlet problem for fully nonlinear elliptic operators}, {\it Adv. Math.} \textbf{218}  (2008), no. 1, 105-135.
 \bibitem{Sha} J. P. Sha,  \emph{Handlebodies and $p$-convexity}, J. Differential Geometry 25 (1987), 353-361.
\bibitem{Wu} H. Wu,  \emph{Manifolds of partially positive curvature},  Indiana Univ. Math. J. 36 (1987),  525-548.

\end{thebibliography}
\end{document}